\documentclass[sn-basic]{sn-jnl}% Basic Springer Nature Reference Style/Chemistry Reference Style
\usepackage{graphicx}%
\usepackage{multirow}%
\usepackage{amsmath,amssymb,amsfonts}%
\usepackage{amsthm}%
\usepackage{mathrsfs}%
\usepackage[title]{appendix}%
\usepackage{xcolor}%
\usepackage{textcomp}%
\usepackage{manyfoot}%
\usepackage{booktabs}%
\usepackage{algorithm}%
\usepackage{algorithmicx}%
\usepackage{algpseudocode}%
\usepackage{listings}%
\usepackage{threeparttable}
\usepackage{float}
\usepackage[utf8]{inputenc}
\usepackage{natbib}
\theoremstyle{thmstyleone}%
\newtheorem{theorem}{Theorem}%  meant for continuous numbers
\theoremstyle{thmstyletwo}%
\newtheorem{lemma}{Lemma}

\theoremstyle{thmstylethree}%
\newtheorem{definition}{Definition}%
\numberwithin{equation}{theorem}
\begin{document}

\title[A NEW SIR MODEL WITH MOBILITY]{FORMULATION  OF A NEW SIR MODEL WITH NON-LOCAL MOBILITY}


\author[1]{\fnm{Ciana} \sur{Applegate}}\email{ciana.applegate@eku.edu}

\author[2]{\fnm{Jiaxu} \sur{Li}}\email{jiaxu.li@louisville.edu}
\equalcont{These authors contributed equally to this work.}

\author*[2]{\fnm{Dan} \sur{Han}}\email{dan.han@louisville.edu}
\equalcont{These authors contributed equally to this work.}

\affil*[1]{\orgdiv{Department of Mathematics and Statistics}, \orgname{Eastern Kentucky University}, \orgaddress{\street{521 Lancaster Ave}, \city{Richmond}, \postcode{40475}, \state{Kentucky}, \country{United States}}}

\affil[2]{\orgdiv{Department of Mathematics}, \orgname{University of Louisville}, \orgaddress{\street{2301 South 3rd Street}, \city{Louisville}, \postcode{40202}, \state{Kentucky}, \country{United States}}}


\abstract{In this manuscript, we develop a mobility-based Susceptible-Infectious-Recovered (SIR) model to elucidate the dynamics of pandemic propagation. While traditional SIR models within the field of epidemiology aptly characterize transitions among susceptible, infected, and recovered states, they typically neglect the inherent spatial mobility of particles. To address this limitation, we introduce a novel dynamical SIR model that incorporates nonlocal spatial motion for three distinct particle types, thereby bridging the gap between epidemiological theory and real-world mobility patterns. This paper primarily focuses on analyzing the long-term behavior of this dynamic system, with specific emphasis on the computation of first and second moments. We propose a new reproduction number $R_0^m$ and compare it with the classical reproduction number $R_0$ in the traditional SIR model. Furthermore, we rigorously examine the phenomenon of intermittency within the context of this enhanced SIR model. The results contribute to a more comprehensive understanding of pandemic spread dynamics, considering both the interplay between disease transmission and population mobility and the impact of spatial motion on the system's behavior over time.}

\keywords{SIR ; SIR model ; 
Mobility ; Epidemic modeling ; Dynamic spatial modeling.}



\maketitle

\section{Introduction}\label{sec1}

The susceptible-infectious-recovered (SIR) model, initially developed by Ronald Ross, William Hamer, and others in the early twentieth century, has been a cornerstone in the mathematical modeling of infectious diseases (\citep{24}). This classical model tracks the sizes of three distinct sub-populations over time: the susceptible population $S(t)$, the infected population $I(t)$, and the recovered population $R(t)$. The traditional model usually consists of three non-spatial coupled nonlinear ordinary differential equations:
\begin{align*}
\displaystyle \frac{dS(t)}{dt} &= - \beta S(t)I(t),\\
\displaystyle \frac{dI(t)}{dt} &= \beta S(t)I(t) - \gamma I(t),\\
\displaystyle \frac{dR(t)}{dt} &= \gamma I(t).
\end{align*} 
where $\beta$ represents the infection rate—the average number of susceptible individuals infected by one infectious individual per contact per unit of time—and $\gamma$ represents the recovery rate—the average number of infected individuals who recover per unit of time. For decades, the deterministic SIR model has served as a foundational framework in epidemiology, with researchers extending it to include additional compartments, leading to models such as the SI, SIS, SEIR, and SEIS frameworks. While these compartmental models have proven useful for understanding large-scale disease dynamics, they often fail to capture the inherent randomness and uncertainty present in the real-world spread of pandemics. In response to this limitation, Linda Allen and colleagues (\citep{allen1994,allen2008}) made  contributions by developing stochastic variants of the SI, SIR, and SIS models, both in discrete and continuous time. These stochastic models employed branching approximation theory and probability-generating functions (pgfs) to analyze long-term behavior near the disease-free equilibrium. One of the key insights from their work was the relationship between the basic reproductive number \(R_0\) and disease persistence. In deterministic models, \(R_0\) governs whether a disease will fade or persist: if \(R_0 \leq 1\), the disease will eventually die out, while if \(R_0 > 1\), it will continue to spread. However, despite these advances, these stochastic models did not account for the role of mobility in disease transmission  (\citep{allen1994,allen2008,keeling2007modeling}), leaving the extremely important aspect of real-world epidemics unexplored. While several extensions of the classical SIR model have incorporated mobility based on the original framework (\citep{gai2020localized, chinviriyasit2010numerical, centres2024diffusion}), these models often rely on simplifying assumptions—such as describing mobility with a single constant intensity parameter or using Brownian motion as the spreading mechanism—that fail to capture the complexities of real-world human movement. Consequently, there is a need for a more comprehensive integration of mobility processes into epidemic models. This paper addresses this gap by introducing the operator $\mathcal{L}$, which accommodates both local and nonlocal movement patterns, providing a more flexible and realistic framework for simulating the spread of infectious diseases.

The deterministic and stochastic SIR-type models have necessitated numerous extensions and adaptations to capture the complexities of real-world disease dynamics. Researchers have developed a wide range of SIR model variants, each tailored to address specific epidemiological challenges and methodologies, as summarized in Table \ref{tab:variants of SIR}. These models span from basic compartmental frameworks to more sophisticated approaches that integrate agent-based simulations, network-based interactions, stochastic processes, spatial dynamics, and Bayesian methods. Within these categories, various extensions have been introduced, including the addition of new compartments, social structures, and multiple disease strains etc. These modifications allow for a more accurate representation of diverse biological conditions and environmental factors, providing a more comprehensive understanding of disease dynamics across different contexts. Furthermore, hybrid models have emerged, combining features from multiple modeling approaches to provide a more comprehensive understanding of disease dynamics under varying conditions (\citep{perez2009agent, wu2018hybrid, strano2017hybrid, zhao2019hybrid, small2015stochastic, pastor2014hybrid, white2016hybrid}).

\begin{table}[h]
\centering
\resizebox{\textwidth}{!}{
\begin{tabular}{|p{3cm}|p{3cm}|p{3cm}|p{3cm}|}
\hline
\textbf{Model Type} & \textbf{Key Features} & \textbf{Use Cases} & \textbf{Limitations} \\ \hline
\textbf{Compartmental Models\footnotemark[1]} & Simple, well-defined transitions between states & Large-scale population dynamics & Assumes homogeneity within compartments \\  \hline
\textbf{Stochastic Models\footnotemark[2]} & Incorporates randomness, simulates multiple scenarios & Small populations, early outbreak stages & Computationally intensive, requires a large number of simulations \\ \hline
\textbf{Network/Contact Models\footnotemark[3]} & Models disease spread through specific contact networks & Diseases spread through specific contact patterns & Requires detailed contact data, static contact network \\ \hline
\textbf{Bayesian Models\footnotemark[4]} & Integrates prior knowledge, probabilistic predictions & Model regional effects and time evolution & Sensitivity to priors, no mobility \\ \hline
\textbf{Agent-Based Models\footnotemark[5]} & Individual-level simulation, heterogeneous agents & Complex systems, social interactions & Computationally expensive, difficult to calibrate \\ \hline
\end{tabular}}
\caption{Comparison of epidemiological SIR models}
\label{tab:variants of SIR}
\end{table}
% Footnotes for the table
% Footnote 1
\footnotetext[1]{\citep{kermack1927contribution, anderson1991infectious}}

% Footnote 2
\footnotetext[2]{\citep{allen2008introduction, keeling2007modeling,allen1994}}

% Footnote 3
\footnotetext[3]{\citep{Liggett2013,pastor2001epidemic, newman2002spread}}

% Footnote 4
\footnotetext[4]{\citep{jewell2009bayesian, oneill1999bayesian}}

% Footnote 5
\footnotetext[5]{\citep{epstein2009modeling, bonabeau2002agent}}

Despite the advances in epidemic modeling, relatively few models have successfully integrated realistic spatial random motions of populations. One common method for incorporating mobility into SIR models is through random diffusive processes. Most research in this area utilizes either Brownian motion driven by the local Laplacian operator $\Delta f(x)=\displaystyle\frac{1}{2}\sum\limits_{i=1}^d \frac{\partial^2 f}{\partial x_i^2}$ in $R^d$ (\citep{centres2024diffusion, faranda2020modeling}) or discrete simple random walks (\citep{hisi2019role}). The primary advantage of these diffusion-based approaches lies in their ability to capture the inherent randomness of movement, providing a more accurate representation of real-world variability in human or animal mobility compared to fixed mobility intensity models. However, their key limitation is their local spatial spreading focus: both Brownian motion and simple random walks mainly model short-range movements to the nearest neighborhoods and fail to account for long-distance travel or nonlocal interactions  In modern contexts, geographic distance does not necessarily correspond to effective distance due to advanced transportation systems; distant locations can function as immediate neighbors because a single location may be connected to multiple others through flight connections or other transport networks. Consequently, movement between locations can be viewed as a diffusion process with weighted transitions, where different probabilities are assigned to connections between locations in a high-dimensional space. This complex connectivity necessitates building more general approaches that account for both spatial geographic local interactions and non-local interactions due to transportation networks like flights, often requiring higher-dimensional spaces like $Z^d$ and $R^d$. As a result, these traditional diffusion models may not fully capture the complexity of mobility in scenarios where long-distance travel or migration plays a critical role in disease spread, as highlighted by \citep{colizza2007modeling, citron2021comparing,balcan2009multiscale}. 

The second common approach to modeling mobility in SIR frameworks is to introduce a mobility intensity parameter or a constant mobility matrix to account for individual movement in one whole system or between two or multiple ($n$) patches.  This approach has been explored in deterministic ODE models, stochastic models, and network models (\citep{wang2003threshold,ahmed2012modeling,chen2014transmission,xie2023contact,ding2012global,bichara2015sis,akuno2023multi,nipa2020disease,abhishek2021epidemic}). To describe the long-distance mobility, these models typically adopt a multi-patch framework, where individuals within each patch follow local SIR dynamics and move between patches based on a class-dependent continuous-time Markov chain. Although this framework facilitates long-distance mobility between different regions or patches, it actually does not involve with respect to spatial distance and generally assumes homogeneous movement among the population and between patches, with a fixed mobility intensity or constant mobility matrix, which oversimplifies the variability and randomness inherent in real human mobility.

To better capture long-distance movement, \citep{centres2024diffusion} proposed an agent-based model where susceptible individuals become infected if they are located within a cutoff radius of an infected individual. This model aims to represent spatial proximity as a key factor in transmission, reflecting how localized and non-localized interactions contribute to the spread of diseases. However, determining an appropriate cutoff radius in real-world scenarios is not straightforward. In practice, human movement is often influenced by social, economic, and environmental factors, leading to irregular patterns that cannot be easily captured by a fixed radius. 

To better explore the spatial space, \citep{paoluzzi2021single} employed an agent-based SIR model on a lattice $Z^2$, where agents' mobility is governed by L\'{e}vy random walks driven by the fractional Laplacian $(-\Delta)^{\beta/2}$ governed by the mobility intensity parameter $0<\beta<2$, allowing for both short-range interactions and occasional long-distance jumps. $\beta$ controls the probability distribution of the step length. As $\beta\rightarrow 2$, the process behaves like Brownian motion, and as $\beta$ decreases to $1$, the probability of long jump increases. While the use of L\'{e}vy flights provides a more flexible and realistic representation of mobility patterns, the fixed mobility intensity $\beta$ limits flexibility. It treats interactions equally in all directions (symmetric) and assumes that these interactions are uniform across the entire space (homogeneous). It enforces a fixed power-law decay of interactions over distance, which may not capture direction-dependent movement or varying mobility patterns shaped by geographic, social, or individual factors.

This motivates us to explore a new nonlocal operator that can more generally simulate movement patterns, regional interactions, and spatial heterogeneity in disease spread. \citep{19} studied a continuous-time branching random walk generated by a nonlocal Laplacian operator on the multidimensional lattice $Z^d$. Their work (\citep{19,25,28}) focused on the asymptotic behavior of particles, using spectral theory to analyze moments for the evolutionary operator. However, their models considered only single-type branching random walks. In this paper, we incorporate this spatial migration mechanism into an epidemic SIR model and formulate a novel model with nonlocal spatial mobility.

In this paper, Section 2 gives a description of the new SIR model with nonlocal spatial mobility. Section 3 consists of the derivation of the stochastic differential equations of the first moments (expectations) of the susceptible, infected and recovered population groups. The section also contains the solutions, and the analysis of the behavior of the first moments of the three groups. This is used to find the expected values of the susceptible, infected, and recovered groups. Section 4 focuses on and analyzes the second moments of the more interesting infected group, which relates to the variance. Section 5 contains the analysis of the intermittency phenomenon and the effected of clusterization of the infected population.

\section{Formulation of SIR model with mobility}

In this paper, a new SIR model with spatial mobility on $\mathbb{Z}^d$ space is introduced. The new model with migration has the assumption that all of the particles can have spatial motion, and transitions among the susceptible, infected, and recovered groups. We assume that $N(t,x) = S(t,x)+ I(t,x)+R(t,x)$, where $N(t,x)$ is the total population at position $x$ at time $t$, is varying. We define $\kappa$ as the probability that a particle will migrate. We define $\beta$ as the transition rate from the susceptible group $S$ to the infected group $I$, and $\gamma$ is the transition rate from infected group $I$ to recovered group $R$. Regarding the migration direction, it is determined by the probability intensity $a(z)$. One particle moving from location $x$ to location $x+z$ has probability $\kappa a(z)dt$ during the infinitesimal time period $(t, t+dt)$. Assume that $a(z)$ satisfies the following conditions:\\
\indent\quad \quad \textcircled{1} $a(z) = a(-z)$, \quad \textcircled{2} $\sum\limits_{z \in \mathbb{Z}^d} a(z)=0$, \quad and \quad \textcircled{3} $\sum\limits_{z \neq 0} a(z)=1$, \quad \\
which implies that $a(0)=-1$. Additionally, we assume that the spatial motion of healthy particles is the same as the spatial motion of an infected particle, and that only one event can happen during an infinitesimally short time period $(t, t+dt)$. In other words, a particle can either jump to another location or they can transmit states. The possible events are:
\begin{enumerate}[(i)]
\item $(S,x) \rightarrow (S,x+z)$ with probability $\kappa a(z)dt, \forall x,z \in \mathbb{Z}^{d}$.\\ %\\
This is the event that in an infinitesimally short time period $(t,t+dt)$, a particle at location $x$ moves to location $x+z$ within the susceptible group.\\
\item $(S,x+z) \rightarrow (S,x)$ with probability $\kappa a(-z)dt, \forall x,z \in \mathbb{Z}^{d}$.\\
This is the event that in an infinitesimally short time period $(t,t+dt)$, a particle at location $x+z$ moves to location $x$ within the susceptible group.\\
\item $(I,x) \rightarrow (I,x+z)$ with probability $\kappa a(z)dt, \forall x,z \in \mathbb{Z}^{d}$.\\
This is the event that in an infinitesimally short time period $(t,t+dt)$, a particle at location $x$ moves to location $x+z$ within the infected group.\\
\item $(I,x+z) \rightarrow (I,x)$ with probability $\kappa a(-z)dt, \forall x,z \in \mathbb{Z}^{d}$.\\
This is the event that in an infinitesimally short time period $(t,t+dt)$, a particle at location $x+z$ moves to location $x$ within the infected group.\\
\item $(R,x) \rightarrow (R,x+z)$ with probability $\kappa a(z)dt, \forall x,z \in \mathbb{Z}^{d}$.\\
This is the event that in an infinitesimally short time period $(t,t+dt)$, a particle at location $x$ moves to location $x+z$ within the recovered group.\\
\item $(R,x+z) \rightarrow (R,x)$ with probability $\kappa a(-z)dt, \forall x,z \in \mathbb{Z}^{d}$.\\
This is the event that in an infinitesimally short time period $(t,t+dt)$, a particle at location $x+z$ moves to location $x$ within the recovered group.\\
\item $(S,x) \rightarrow (I,x)$ with probability $\beta dt, \forall x,z \in \mathbb{Z}^{d}$.\\
This is the event that a particle at location $x$ in the susceptible group transitions to the infected group.\\
\item $(I,x) \rightarrow (R,x)$ with probability $\gamma dt, \forall x,z \in \mathbb{Z}^{d}$.\\
This is the event that a particle at location $x$ in the infected group transitions to the recovered group.
\end{enumerate}

\section{The First Moments}

One of the important questions in epidemic models is what are the expected values of the susceptible, infected, and recovered population groups and what is the long term asymptotic behavior. 

\subsection{Deriviation of the Differential Equations}
In this section, the Kolmogorov forward stochastic differential equations are derived and used to find the first moments (expectations) of the three population groups.

\begin{theorem}\label{thm: differential equations first moment}
The differential equations for the first moment of the susceptible, infected, and recovered groups are 
\begin{align}\label{eq: diff eq first moment sus}
\displaystyle \frac{\partial m_{1}^{S}(t,x)}{\partial t} &= \kappa \mathcal{L} m_{1}^{S}(t,x) - \beta m_{1}^{I}(t,x),\\
\label{eq: diff eq first moment inf}
\displaystyle \frac{\partial m_{1}^{I}(t,x)}{\partial t} &= \kappa \mathcal{L} m_{1}^{I}(t,x) + (\beta - \gamma) m_{1}^{I}(t,x),\\
\label{eq: diff eq first moment rec}
\displaystyle \frac{\partial m_{1}^{R}(t,x)}{\partial t} &= \kappa \mathcal{L} m_{1}^{R}(t,x) + \gamma m_{1}^{I}(t,x).
\end{align}
where $m_{1}^{D}(t,x) = E[D(t,x)]$ for $D= S, I,$ or $R$ and the discrete Laplace operator is the generator of the underlying random walk of the individuals and $$\mathcal{L}f(t,x) = \sum\limits_{z \neq 0, z \in \mathbb{Z}^{d}} a(z)[f(t, x+z)- f(t,x)], x \in \mathbb{Z}^{d}.$$
\end{theorem}

\begin{proof} During the infinitesimal time period $(t,t+dt)$, for the susceptible group at location $x \in \mathbb{Z}^{d}$ is $S(t+dt,x)= S(t,x) + \xi (dt)$ where\\ 
$$\xi(dt) = \begin{cases}
        1 \hspace{.25cm} w.p \hspace{.25cm} \sum\limits_{z \neq 0}S(t, x+z) \kappa a(z) dt \hspace{.25cm} \text{for the event (ii)}\\
        -1 \hspace{.25cm} w.p \hspace{.25cm} \sum\limits_{z \neq 0}S(t, x) \kappa a(z) dt + I(t,x) \beta dt \hspace{.25cm} \text{for the events (i),(vii)}\\
        0 \hspace{.25cm} \text{otherwise,}
    \end{cases}$$
and    
\begin{align*}
    \qquad m_{1}^{S}(t+dt,x)]&= E \Big[(S(t,x)+1)[\sum\limits_{z \neq 0}S(t, x+z) \kappa a(z) dt]\\ 
    &+ (S(t,x)-1) [\sum\limits_{z \neq 0}S(t, x) \kappa a(z) dt+ I(t,x) \beta dt] \\
    &+ S(t,x)[1- \sum\limits_{z \neq 0}S(t, x+z) \kappa a(z) dt- (\sum\limits_{z \neq 0}S(t, x) \kappa a(z) dt + I(t,x) \beta dt)\Big]
\end{align*}

Subtract $m_1^S(t,x)$ and divide $dt$ at both sides, $dt\rightarrow 0$, we can get Equation \eqref{eq: diff eq first moment sus}:
\begin{align*}
    \displaystyle \frac{\partial m_{1}^{S}(t,x)}{\partial t} = \kappa \mathcal{L} m_{1}^{S}(t,x) - \beta m_{1}^{I}(t,x).
\end{align*}

For the infected group: $I(t+dt,x)= I(t,x) + \xi (dt)$, where

 $\xi (dt)= \begin{cases}
        1, \hspace{.25cm} w.p \hspace{.25cm} I(t,x) \beta dt + \sum\limits_{z \neq 0}I(t, x+z) \kappa a(-z) dt \hspace{.25cm} \text{for the events (iv), (vii)}\\
        -1, \hspace{.25cm} w.p \hspace{.25cm} I(t,x) \gamma dt + \sum\limits_{z \neq 0}I(t, x) \kappa a(z) dt \hspace{.25cm} \text{for the events (iii), (viii)}\\
        0, \hspace{.25cm} \text{otherwise,}
    \end{cases}$

Utilizing a similar process for the infected group, we get Equation \eqref{eq: diff eq first moment inf}:\\

$\displaystyle \frac{\partial m_{1}^{I}(t,x)}{\partial t} = \kappa \mathcal{L} m_{1}^{I}(t,x) + (\beta - \gamma) m_{1}^{I}(t,x)$.

For the recovered group: $R(t+dt,x)= R(t,x) + \xi (dt)$ where

$\xi(dt)= \begin{cases}
        1 \hspace{.25cm} w.p \hspace{.25cm} I(t,x) \gamma dt + \sum\limits_{z \neq 0}R(t, x+z) \kappa a(-z) dt \hspace{.25cm} \text{for the events (vi),(viii)}\\
        -1 \hspace{.25cm} w.p \hspace{.25cm} \sum\limits_{z \neq 0}R(t, x) \kappa a(z) dt \hspace{.25cm} \text{for the event (v)}\\
        0 \hspace{.25cm} \text{otherwise}
    \end{cases}$\\
    
Similar to the first two population groups, we get Equation \eqref{eq: diff eq first moment rec}:\\

$\displaystyle \frac{\partial m_{1}^{R}(t,x)}{\partial t} = \kappa \mathcal{L} m_{1}^{R}(t,x) + \gamma m_{1}^{I}(t,x)$.\\

\end{proof}
To elucidate the differences between the nonlocal operator $\kappa\mathcal{L}$ and the Laplace operator $D\Delta$ resulted from Brownian motion, we compare the
model in Theorem \ref{thm: differential equations first moment} with the current popular diffusion model (\citep{chinviriyasit2010numerical}). The structures in subsequent studies that extend this model resemble a similar format driven by the generator of Brownian motion, $\Delta$ (the Laplacian operator), as seen in works like (\citep{caraballo2018analysis,ammi2023optimal}). Let $\Omega$ be a bounded domain in $R^d$ with smooth boundary $\partial \Omega$ and $\eta$ be the outward unit normal vector on the boundary, the SIR reaction-diffusion model can be described by
\begin{align*}
& S_t=D \Delta S-\beta S I, \quad\qquad z \in \Omega, t>0, \\
& I_t=D \Delta I-\gamma I+\beta S I, \quad z \in \Omega, t>0, \\
& R_t=D \Delta I+\gamma I, \quad\qquad \ \ z \in \Omega, t>0,
\end{align*}

with homogeneous Neumann boundary conditions
$$
\partial_\eta S=\partial_\eta I=\partial_\eta R=0, \quad z \in \partial \Omega, t>0,
$$

and initial conditions
$$
S(z, 0)=S_0(z) \geqslant 0, \quad I(z, 0)=I_0(z) \geqslant 0, \quad R(z, 0)=R_0(z) \geqslant 0, \quad z \in \bar{\Omega},
$$

where $S(z, t), I(z, t)$ and $R(z, t)$ denote the numbers of susceptible, infected and recovered individuals at location $z$ and time $t$. $N$ is the total population. The transmission from susceptibles to infectives is assumed to be $\beta S I$ where $\beta$ is the transmissibility coefficient. The spatial propagation of the individuals is modeled by the constant diffusion coefficients $D_S \geqslant 0, D_I \geqslant 0$ and $D_R \geqslant 0$ for the susceptibles, infected and recovered, respectively. And here $D=D_S=D_I=D_R$. The homogeneous Neumann boundary condition ensures that the model respects spatial containment, aligning with the idea of a closed system.

 Unlike Brownian motion, which is governed by the Laplacian operator $\Delta f(x)=\frac{1}{2} \sum_{i=1}^d \frac{\partial^2 f}{\partial x_i^2}$ in continuous space $\mathbb{R}^d$ in (\citep{chinviriyasit2010numerical,ammi2023optimal,caraballo2018analysis}) and arises as the scaling limit of discrete random walks restricted to local neighborhood movements, the operator $\mathcal{L}$ allows for a broader class of movement patterns.  In contrast, the operator $\mathcal{L}$ allows for a broader class of movement patterns. Specifically, $\mathcal{L}$ enhances epidemic models by incorporating complex mobility patterns that accommodate both local and nonlocal movements.

To compare $\kappa\mathcal{L}$ with $D\Delta$ , we consider the continuum limit of the discrete model. The discrete nonlocal operator $\mathcal{L}$ can approximate the continuous Laplacian $\Delta$ under appropriate scaling of $\kappa$ and $a(z)$. If $a(z)$ is significantly supported only for $z$ in a neighborhood with limit radius, that is, when the mobility kernel $a(z)$ is symmetric and has finite second moments (i.e., finite variance $\sigma^2=\sum\limits_{z\in Z^d}|z|^2a(z)<\infty$), then the movements are predominantly within certain range, the nonlocal operator $\mathcal{L}$ can be approximated by a multiple of Laplacian operator $\Delta$. This is achieved via a Taylor expansion and leveraging the properties of the kernel.

For smooth functions $f(t,x)$, we can expand $f(t,x+hz)$ around $x$ using a Taylor series where the lattice spacing $h\rightarrow 0$:
$$
f(t,x+hz)=f(t,x)+hz^{\top} \nabla f(t,x)+\frac{1}{2} h^2z^{\top} \nabla^2 f(t,x) z+\cdots .
$$

Substituting this expansion into the operator $\mathcal{L} f(t,x)$, we have:
$$
\mathcal{L} f(t,x)=\sum_{z \in \mathbb{Z}^d}[f(t,x+hz)-f(t,x)] a(z) \approx \sum_{z \in \mathbb{Z}^d}\left(hz^{\top} \nabla f(t,x)+\frac{1}{2}h^2z^{\top} \nabla^2 f(t,x) z\right) a(z) .
$$

Because the mobility kernel $a(z)$ is symmetric (i.e., $a(z)=a(-z)$), $\sum\limits_{z\in Z^d} z a(z)=0$ and thus the first-order term vanishes:
$$
\sum_{z \in \mathbb{Z}^d} z^{\top} \nabla f(t,x) a(z)=\nabla f(t,x)^{\top} \sum_{z \in \mathbb{Z}^d} z a(z)=0
$$
Therefore, the leading term is:
$$
\mathcal{L} f(t,x) \approx \sum_{z \in \mathbb{Z}^d} \frac{h^2}{2} z^{\top} \nabla^2 f(t,x) z a(z)
$$

This expression can be rewritten using the properties of the trace and assuming isotropy if we want to approximate the Brownian motion or simple random walk in which all directions are equivalent:
$$
\mathcal{L} f(t,x) \approx \frac{h^2}{2} \operatorname{Tr}\left(\nabla^2 f(t,x) \sum_{z \in \mathbb{Z}^d} z z^{\top} a(z)\right) .
$$

Let $\Sigma$ be the covariance matrix of the kernel $a(z)$ :$
\Sigma=\sum_{z \in \mathbb{Z}^d} z z^{\top} a(z) .
$

If $a(z)$ is isotropy (movement is the same in all directions), $\Sigma=\sigma^2 I$, where $\sigma^2$ is the variance:
$$
\sigma^2=\sum_{z \in \mathbb{Z}^d}|z|^2 a(z) .
$$

Therefore, we have:
$$
\mathcal{L} f(t,x) \approx \frac{h^2}{2} \operatorname{Tr}\left(\nabla^2 f(t,x) \sigma^2 I\right)=\frac{\sigma^2}{2} \operatorname{Tr}\left(\nabla^2 f(t,x)\right)=\frac{h^2\sigma^2}{2} \Delta f(t,x),
$$

In other words, the diffusion coefficient in the classical stochastic model driven by Brownian motion $D\approx \displaystyle\frac{\kappa h^2\sigma^2}{2}$.
As the lattice spacing $h$ approaches zero, the diffusion coefficient $D=\frac{\kappa h^2 \sigma^2}{2}$ also approaches zero, assuming that $\kappa$ and $\sigma^2$ remain constant. This seems counterintuitive because we expect the diffusion coefficient $D$ in the classical model to remain finite to describe meaningful diffusion in the continuous limit.

The key to resolving this issue lies in understanding how the parameters $\kappa, h$, and $\sigma^2$ should scale relative to each other when transitioning from a discrete to a continuous model. The key to resolving this issue lies in understanding how the parameters $\kappa, h$, and $\sigma^2$ should scale relative to each other when transitioning from the current nonlocal discrete model to a continuous model. To maintain a finite diffusion coefficient $d$ as $h \rightarrow 0$, we need to adjust the mobility rate $\kappa$ appropriately. Specifically, $\kappa$ should scale inversely with $h^2$ :
$$
\kappa=\frac{\tilde{\kappa}}{h^2},
$$

where $\tilde{\kappa}$ is a constant independent of $h$. As the lattice becomes finer (i.e., $h$ becomes smaller), individuals need to take more steps to cover the same macroscopic distance in a given time. If we keep $\kappa$ constant, individuals would move less overall as $h$ decreases, leading to $d \rightarrow 0$, which implies no diffusion in the continuous limit.

By scaling $\kappa$ inversely with $h^2$, we ensure that the overall movement of individuals remains consistent as $h \rightarrow 0$. This adjustment compensates for the decreasing step size by increasing the movement frequency, thereby maintaining the same nonlocal diffusion behavior. That is, the nonlocal discrete operator $\mathcal{L}$ approximates the continuous Laplacian $\Delta$ when the movements are local and the lattice spacing is small.

To illustrate the impact of different mobility patterns on epidemic spread, we compare the epidemic distributions generated by the nonlocal operator $\mathcal{L}$ and a local Brownian motion type  simple random walk on a $51$ by $51$ 
 $Z^2$ grid. We assume that a single individual is initially infected at the central location $(26,26)$. In the first scenario, the nonlocal operator $\mathcal{L}$ utilizes a mobility kernel $a(z)$ that follows a normal distribution with mean 0 and variance 16, allowing individuals to traverse longer distances representing nonlocal movement. In the second scenario, we consider a simple random walk where individuals have an equal probability of moving to neighboring locations only, reflecting local movement patterns akin to Brownian motion.
\begin{figure}[h]
    \centering
    \begin{minipage}[b]{0.48\linewidth}
        \centering
        \includegraphics[width=\linewidth]{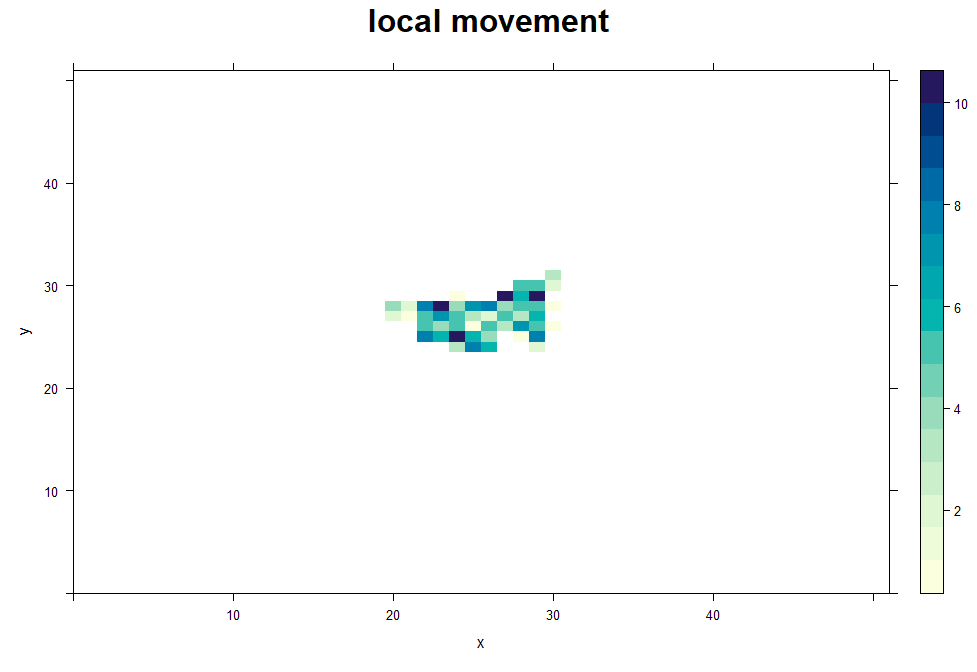}
        \label{fig:local}
    \end{minipage}\hfill
    \begin{minipage}[b]{0.48\linewidth}
        \centering
        \includegraphics[width=\linewidth]{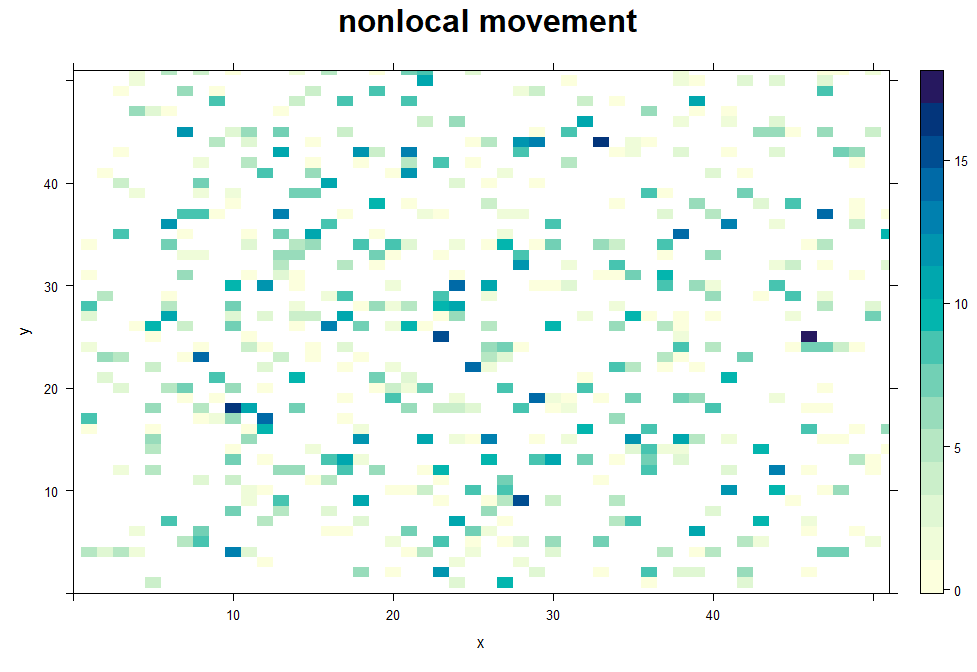}
        \label{fig:nonlocal}
    \end{minipage}
    \caption{This comparison highlights the critical differences in epidemic spread dynamics under local versus nonlocal movement patterns. Over 50,000 iterations, local movement patterns driven by Brownian motion or simple random walk tend to confine the spread within smaller areas, often resulting in centralized outbreaks, whereas our nonlocal modeling approach facilitates a pattern in wide and more dispersed distributions of infections. }
    \label{fig:both}
\end{figure}

As shown in Figure \ref{fig:both}, the epidemic distributions differ significantly between the two movement patterns. The non-local movement modeled by the normal distribution allows individuals to traverse longer distances, resulting in a wider spread of the infection. In contrast, the local movement pattern restricts individuals to nearby locations, leading to a more concentrated epidemic distribution. This comparison demonstrates the effectiveness and flexibility utilizing the operator $\mathcal{L}$ to model various mobility patterns by adjusting the kernel $a(z)$. The comparison suggests that the operator $\mathcal{L}$ provides a more flexible framework in simulating epidemic spread than the models solely based on Brownian motion, which are generally used in the SIR models and its variants.

By adjusting the kernel $a(z)$, we can simulate how individuals move and spread disease over both short and long distances. This concept extends naturally from the discrete lattice $\mathbb{Z}^d$ to continuous space $\mathbb{R}^d$ by defining a similar operator:
\[
\mathcal{L} f(t, x) = \int_{\mathbb{R}^d} a(z) [f(t, x + z) - f(t, x)] \, dz, \quad x \in \mathbb{R}^d.
\]

In particular, when $a(z)$ has finite variance, that is, $\int_{\mathbb{R}^d} |z|^2 a(z) \, dz < \infty$, the corresponding random walk approximates Brownian motion under scaling limits. Conversely, if $a(z)$ has infinite variance, meaning $\int_{\mathbb{R}^d} |z|^2 a(z) \, dz = \infty$, the random walk approximates a stable Lévy process under scaling limits, according to the generalized central limit theorem.

Therefore, by carefully selecting the kernel $a(z)$, the operator $\mathcal{L}$ enables us to model complex mobility patterns, including heavy-tailed movement distributions that capture both local and long-range human mobility. This provides a more accurate representation of real-world human movement in epidemic models, allowing for the simulation of disease spread over various spatial scales and enhancing our understanding of epidemic dynamics.

\subsection{Solving for the First Moments}

Now in order to solve the differential equations there are 2 cases- homogeneous space and inhomogeneous space. The first case is the homogeneous space where there is the assumption the spaces are equivalent, meaning $x$ and $x+z$ are the same and the Laplace operator $\mathcal{L} f(t,x)= \sum\limits_{z \neq 0} a(z)[f(t,x+z)-f(t,x)]=0$. Then the differential equations from Theorem \ref{thm: differential equations first moment} now no longer have the Laplace operator and the differential equations become:
\begin{align}
    \label{eq: sus diff equation from thm 1 without L}
    \displaystyle \frac{\partial m_{1}^{S}(t,x)}{\partial t} &= - \beta m_{1}^{I}(t,x),\\
    \label{eq: inf diff equation from thm 1 without L}
    \displaystyle \frac{\partial m_{1}^{I}(t,x)}{\partial t} &= (\beta - \gamma) m_{1}^{I}(t,x),\\
    \label{eq: rec diff equation from thm 1 without L}
    \displaystyle \frac{\partial m_{1}^{R}(t,x)}{\partial t} &= \gamma m_{1}^{I}(t,x).
\end{align}

\begin{theorem}\label{thm:first moment homogeneous space}
In the homogeneous space, as $t \longrightarrow \infty$, with initial conditions $S(0,x) = \rho_{0} >0, I(0,x) = \delta_y(x)=\begin{cases}
        1 \hspace{.25cm} if \hspace{.25cm} x=y\\
        0 \hspace{.25cm} if \hspace{.25cm} x \neq y
\end{cases}$, the location $y$ marks the starting point of the epidemic, where the first confirmed case of infection has been identified. $R(0,x)=0$, if $\beta \neq \gamma$, the steady states $m_{1}^{S}(t,x)$, $m_{1}^{I}(t,x)$, and $m_{1}^{R}(t,x)$ exist and the solutions are 
\begin{align}
m_{1}^{S}(t,x) &= \displaystyle \frac{-\beta e^{\displaystyle (\beta -\gamma)t} + \beta + \rho_{0}(\beta - \gamma)}{\beta - \gamma}\delta_y(x),\\
m_{1}^{I}(t,x) &= e^{\displaystyle (\beta -\gamma)t}\delta_y(x),\\
m_{1}^{R}(t,x) &= \displaystyle \frac{\gamma e^{\displaystyle (\beta - \gamma)t} - \gamma}{(\beta - \gamma)}\delta_y(x).
\end{align}
If $\beta = \gamma$, the solutions are
\begin{align}
m_{1}^{S}(t,x) &= (\rho_{0}- \beta  t)\delta_y(x),\\
m_{1}^{I}(t,x) &= \delta_y(x) ,\\
m_{1}^{R}(t,x) &= \gamma  t \delta_y(x).
\end{align}
\end{theorem}

\begin{proof} Solving the ODE equations \eqref{eq: sus diff equation from thm 1 without L}, \eqref{eq: inf diff equation from thm 1 without L}, \eqref{eq: rec diff equation from thm 1 without L} using regular ODE methods, we get the solutions for the first moments of the SIR model.
\end{proof}

Now we need to solve the differential equations given in Theorem \ref{thm: differential equations first moment} in the inhomogeneous space. With the assumption that the space is inhomogeneous, we have that the spaces $x$ and $x+z$ are not equivalent and $\mathcal{L} f(t,x)= \sum\limits_{z \neq 0} a(z)[f(t,x+z)-f(t,x)] \neq 0$. To solve the in-homogeneous equations, we need a few new definitions and theorems:
    
\begin{definition}[Fourier Transform] \label{def: fourier transform} Denote the Fourier transform of $f(x)$ as
$\displaystyle \hat{f}(k) = \sum\limits_{x \in \mathbb{Z}^d} f(x)e^{\displaystyle ikx}.$
\end{definition}

\begin{definition}[Inverse Fourier Transform] \label{def: inverse fourier transform}
For the Fourier transform $\hat{f}(k) = \sum\limits_{x \in \mathbb{Z}^d} e^{\displaystyle itx} f(x)$, the inverse Fourier transform is $f(x)= \displaystyle \frac{1}{(2 \pi)^{d}} \int_{T^d} \hat{f}(k) e^{\displaystyle -ikx}dk$, where $T^{d} = [- \pi, \pi]^{d}.$
\end{definition}

\begin{definition}[Intensity of Mobility Effect] \label{def: intensity of mobility}
Define $ \alpha=\kappa \hat{a}(k)$, where\\
$\hat{a}(k) = \sum\limits_{z \in \mathbb{Z}^{d}} a(z)e^{\displaystyle ikz}$, as a measure of the intensity of the dynamical movement of particles or mobility effect.
\end{definition}

\begin{lemma}[Fourier Transform of function $\mathcal{L}f(k)$] \label{lemma: fourier transform of L}
Denote $\hat{\mathcal{L}}(k)$ as the Fourier symbol of the operator $\mathcal{L}$, then $\widehat{\mathcal{L}f}(k) = \hat{f}(k)\hat{\mathcal{L}}(k)$ and $\hat{\mathcal{L}}(k) = \hat{a}(k) \le 0$.
\end{lemma}

Now we can use Definitions \ref{def: fourier transform} and \ref{def: inverse fourier transform} to solve for the first moment of $I(t,x)$ in the inhomogeneous space. For the differential equation of $m_{1}^{I}(t,x)$ from Theorem \ref{thm: differential equations first moment} Equation \eqref{eq: diff eq first moment inf}, we can apply the Fourier transform from Definition \ref{def: fourier transform} to both sides of the equation and get:
$\displaystyle \frac{\partial \hat{m}_{1}^{I}(t,k)}{\partial t} = \kappa \hat{\mathcal{L}}(k) \hat{m}_{1}^{I}(t,k) + (\beta - \gamma) \hat{m}_{1}^{I}(t,k)$.

Assume $m_{1}^{I}(0,x) =
\begin{cases}
1 \hspace{.25cm} if \hspace{.25cm} x=0\\
0 \hspace{.25cm} if \hspace{.25cm} x \neq 0
\end{cases}$, then $\hat{m}_{1}^{I}(0,k) = \sum\limits_{x \in \mathbb{Z}^d} m_{1}^{I}(0,x) e^{\displaystyle ikx} = 1$.

$\hat{m}_{1}^{I}(t,k) =  \hat{m}_{1}^{I}(0,k) e^{\displaystyle \big[\kappa \hat{\mathcal{L}}(k) + (\beta - \gamma)\big]t} = e^{\displaystyle \big[\kappa \hat{\mathcal{L}}(k) + (\beta - \gamma)\big]t}$. 

By applying the Inverse Fourier transform from Definition \ref{def: inverse fourier transform}, we obtain:
$$m_{1}^{I}(t,x) = \displaystyle \frac{1}{(2 \pi)^{d}} \int_{T^{d}} \hat{E}[I(t,k)] e^{\displaystyle -ikx} dk = \displaystyle \frac{1}{(2 \pi)^{d}} \int_{T^{d}} e^{\displaystyle [\kappa \hat{\mathcal{L}}(k) + (\beta - \gamma)]t} e^{\displaystyle -ikx} dk.$$

To determine the first moments of the susceptible and recovered groups in inhomogeneous space, it is essential to derive more general solutions that account for spatial variability and the interactions between different population groups. A crucial first step in this process is to establish the transition probability from one location 
$x$ to another location $y$. Understanding this transition probability allows us to model how individuals move through space, which directly influences the spread of the infection and the dynamics of the susceptible and recovered populations.
\begin{lemma} \label{lemma: transition probability}
The transition probability of the particles $p(t,x,y)$ is the fundamental solution to the following equation
\begin{align}
\displaystyle \frac{\partial \mathnormal{p}(t,x,y)}{\partial t} &= \kappa \mathcal{L} \mathnormal{p} (t,x,y),\\
\mathnormal{p}(0,x,y) &= \delta (x-y) = \begin{cases} 
1 \hspace{.25cm} if \hspace{.25cm} x=y\\
0 \hspace{.25cm} if \hspace{.25cm} x \neq y
\end{cases}
\end{align}
and $\mathnormal{p}(t,x,y) = \displaystyle \frac{1}{(2 \pi)^d} \int_{T^{d}} e^{\displaystyle ik(x-y)} e^{\displaystyle \kappa \hat{\mathcal{L}}(k) t} dk.$
\end{lemma}

\begin{proof}
Then the Fourier transform of $\displaystyle \frac{\partial \mathnormal{p}(t,x,y)}{\partial t}$ becomes
\begin{align}
\displaystyle \frac{\partial \hat{\mathnormal{p}}(t,x,k)}{\partial t} &= \kappa \hat{\mathcal{L}}(k) \hat{\mathnormal{p}}(t,x,k),\\
\hat{\mathnormal{p}}(0,x,k) &= \sum\limits_{y \in \mathbb{Z}^{d}} \delta(x-y) e^{\displaystyle iky} = e^{\displaystyle ikx}.
\end{align}
$\implies \hat{\mathnormal{p}}(t,x,k)= \hat{\mathnormal{p}}(0,x,k) e^{\displaystyle \kappa \hat{\mathcal{L}}(k) t}$ and we have $\hat{\mathnormal{p}}(t,x,k)= e^{\displaystyle ikx} e^{\displaystyle \kappa \hat{\mathcal{L}}(k) t}$.\\
Using the Inverse Fourier formula from Definition \ref{def: inverse fourier transform}, we have that\\
$p(t,x,y) = \displaystyle \frac{1}{(2 \pi)^{d}} \displaystyle \int_{T^{d}} \big[ e^{\displaystyle ikx} e^{\displaystyle \kappa \hat{\mathcal{L}}(k) t} \big] e^{\displaystyle -iky}dk = \displaystyle \frac{1}{(2 \pi)^d} \int_{T^{d}} e^{\displaystyle ik(x-y)} e^{\displaystyle \kappa \hat{\mathcal{L}}(k) t} dk$.\\

This means that the transition probability depends on the distance between location $x$ and location $y$. Then plugging $x=0$ and $t=t-s$ into ${p}(t,x,y)$ we have that\\

$p(t-s,0,x-z) = \displaystyle \frac{1}{(2 \pi)^{d}} \int_{T^{d}} e^{\displaystyle ik(z-x)} e^{\displaystyle \kappa \hat{\mathcal{L}}(k)(t-s)}dk$.
\end{proof}

To study the movement properties of the individuals and whether the movement is transient or recurrent, we need to study the Green function:
\begin{center}
$G_{\lambda}(x,y) = \displaystyle \int_{0}^{\infty} e^{\displaystyle -\lambda t} \mathnormal{p}(t,x,y)dt$ for $\lambda \ge 0$.
\end{center}
When $\lambda=0$, $G_{0}(x,y)$, is the expected value of the number of visits of susceptible, infected, or recovered population to location $y$ if the original location is $x$. Either all states are transient or all states are recurrent (either positive recurrent or non-recurrent). When $\lambda =0, x=0, y=0$, the Green function equation becomes $G_{0}(0,0) = \displaystyle \frac{1}{(2 \pi)^{d}} \int_{T^{d}} \displaystyle \frac{1}{- \kappa \hat{\mathcal{L}}(k)} dk$. If $G_{0}(0,0) < \infty$, then the random walk is transient, which means that there is a positive probability that the particle will never return to the original location. If $G_{0}(0,0) = \infty$, then the random walk is recurrent, meaning the infected particle will return to the original location infinitely many times. If the random walk is recurring, then the infected particle can continue to infect the susceptible population at its original location. In this dynamical epidemic model, whether the epidemic at location $x$ will vanish or not depends on whether the infected particle returns original location.

\begin{theorem}[General Solution for Inhomogeneous Equation] \label{thm: general solution}
The general solution to the inhomogeneous equation
\begin{align}
\label{eq: differential general equation}
\displaystyle \frac{\partial U(t,x)}{\partial t} &= \kappa \mathcal{L} U(t,x) + V(t) U(t,x) + f(t,x),\\
U(0,x) &= \rho_{0} > 0,
\end{align}
is $U(t,x) = \rho_{0} E_{x}\Bigg[ e^{\displaystyle \int_{0}^{t} V(X_{s})ds}\Bigg]  + e^{\displaystyle \int_{0}^{t} V(X_{s})ds} \displaystyle \int_{0}^{t} \sum\limits_{z \in \mathbb{Z}^d} \mathnormal{p}(t-s,0,x-z) \tilde{f}(s,z) ds$ where $X_{s}$ is the stochastic process random walk driven by $\kappa \mathcal{L} f = \kappa \sum\limits_{z \in \mathbb{Z}^d} a(z) [f(x+z) -f(x)]$ and $V(X_{s})$ is the potential of the equation.
\end{theorem}

\begin{proof}
By Duhamel's principle, we define the solution of equation \eqref{eq: differential general equation} to be $U(t,x) = U^{h}(t,x) + w(t,x)$, where $U^{h}(t,x)$ is the corresponding homogeneous solution and $w(t,x)$ is a particular solution to the in-homogeneous equation, and we will solve for $U^{h}(t,x)$ and $w(t,x)$ separately.\\

For corresponding homogeneous solution $U^{h}[(t,x)]$, we have
\begin{align}
\displaystyle \frac{\partial U^{h}(t,x)}{\partial t} &= \kappa \mathcal{L} U^{h}(t,x) + V(t) U^{h}(t,x),\\
U^{h}(0,x) &= \rho_{0} > 0,
\end{align}
and the solution to this differential equation is $U^{h}(t,x) = \rho_{0}  E_{x} \Bigg[e^{\displaystyle \int_{0}^{t} V(X_{s}) ds}\Bigg]$ by the Kac-Feyman Formula, where $X_s$ is the stochastic process random walk driven by $
\kappa \mathcal{L} f = \kappa \sum\limits_{z \in \mathbb{Z}^d} a(z)[f(x+z) - f(x)]$.\\

For the particular solution $w(t,x)$, we have
\begin{align}
\label{eq: particular solution for w(t,x)}
\displaystyle \frac{\partial w(t,x)}{\partial t} &= \kappa \mathcal{L} w(t,x) + V(t)w(t,x) + f(t,x),\\
w(0,x) &= 0,
\end{align}
and the solution to this differential equation is $w(t,x) = e^{\displaystyle \int_{0}^{t} V(X_{s})ds} \tilde{w}(t,x)$. Now we need to solve for $\tilde{w}(t,x)$ where
\begin{align}
\displaystyle \frac{\partial w(t,x)}{\partial t} &= \displaystyle \frac{\partial}{\partial t} \Bigg[e^{\displaystyle \int_{0}^{s} V(s)ds} \Bigg] \tilde{w} (t,x) + e^{\displaystyle \int_{0}^{s} V(s)ds} \Bigg[ \displaystyle \frac{\partial \tilde{w} (t,x)}{ \partial t} \Bigg],\\
\label{eq: differential eq for w(t,x)}
&= \Big(\displaystyle \int_{0}^{t}V(s)ds \Big) e^{\displaystyle \int_{0}^{t}V(s)ds} \tilde{w} (t,x) + e^{\displaystyle \int_{0}^{s} V(s)ds} \Bigg[ \displaystyle \frac{\partial \tilde{w} (t,x)}{ \partial t} \Bigg].
\end{align}
Setting equation \eqref{eq: particular solution for w(t,x)} equal to \eqref{eq: differential eq for w(t,x)} we get\\
$\kappa \mathcal{L} w(t,x) + V(t)w(t,x) + f(t,x) = V(t) e^{ \int_{0}^{t}V(s)ds} \tilde{w} (t,x) + e^{ \int_{0}^{s} V(s)ds} \Bigg[ \displaystyle \frac{\partial \tilde{w} (t,x)}{ \partial t} \Bigg]$.\\
Recall that $w(t,x) = e^{\displaystyle \int_{0}^{t} V(X_{s})ds} \tilde{w}(t,x)$, meaning we can cancel terms to get\\
$e^{\displaystyle \int_{0}^{t} V(s)ds} \left[ \displaystyle \frac{\partial \tilde{w} (t,x)}{\partial t} \right] = e^{\displaystyle \int_{0}^{t} V(s)ds} \kappa \mathcal{L} \tilde{w} (t,x) + f(t,x)$.\\
Then for $\tilde{w} (t,x)$ we have
\begin{align}
\displaystyle \frac{\partial \tilde{w}(t,x)}{\partial t} &= \kappa \mathcal{L} \tilde{w}(t,x) + f(t,x) e^{\displaystyle - \int_{0}^{t} V(s)ds},\\
\tilde{w}(0,x) &= 0.
\end{align}
To solve this we use the Fourier Transform from Definition \ref{def: fourier transform} and the transition probability from Lemma \ref{lemma: transition probability} and we find that

$\tilde{w}(t,x) = \displaystyle \int_{0}^{t} \sum\limits_{z \in \mathbb{Z}^{d}} \mathnormal{p} (t-s, 0, x-z) \tilde{f}(s,z)ds$.\\

Plugging $\tilde{w}(t,x)$ into $w(t,x) = e^{\displaystyle \int_{0}^{t} V(X_{s})ds} \tilde{w}(t,x)$ we get that

$w(t,x) = e^{\displaystyle \int_{0}^{t} V(X_{s})ds} \displaystyle \int_{0}^{t} \sum\limits_{z \in \mathbb{Z}^{d}} \mathnormal{p} (t-s, 0, x-z) \tilde{f}(s,z)ds$.\\

Now we have the solution to the general differential equation to be $U(t,x) = U^{h}(t,x) + w(t,x)$. Substituting in what we have for $U^{h}(t,x)$ and $w(t,x)$ we have

$U(t,x) = \rho_{0} E_{x} \left[ e^{\displaystyle \int_{0}^{t} V(X_{s}) ds} \right] + e^{\displaystyle \int_{0}^{t} V(s)ds} \displaystyle \int_{0}^{t} \sum\limits_{z \in \mathbb{Z}^{d}} \mathnormal{p} (t-s, 0, x-z) \tilde{f}(s,z)ds$.\\
\end{proof}
To solve for $m_{1}^{S}(t,x)$ in the inhomogeneous space, we apply Theorem \ref{thm: general solution} where $f(t,x)= - \beta m_{1}^{I}(t,x)$ but there is no potential $V(t)$ in these equations so then we have the solution to be
\begin{align*}
m_{1}^{S}(t,x) &= \displaystyle \frac{1}{(2 \pi)^d} \int_{T^{d}} \rho_{0} e^{\displaystyle \kappa \hat{\mathcal{L}}(k) t} e^{\displaystyle -ikx}dk\\
&- \displaystyle \frac{\beta}{(2\pi)^d} \displaystyle \int_{0}^{t} \sum\limits_{z \in \mathbb{Z}^{d}} \mathnormal{p} (t-s, 0, x-z) \left(\displaystyle \int_{T^{d}} e^{\displaystyle [\kappa \hat{\mathcal{L}}(k) + (\beta - \gamma)]s} e^{\displaystyle -ikz} dk \right) ds.
\end{align*}
The process for solving for the first moments of the recovered group in inhomogeneous space follows a similar to process as the susceptible group, except $f(t,x) = \gamma m_{1}^{I}(t,x)$, and we have the solution to be $m_{1}^{R}(t,x) = m_{1}^{R^{h}}(t,x) + w(t,x)$. Applying Theorem \ref{thm: general solution} to the inhomogeneous differential equations in Theorem \ref{thm: differential equations first moment} we get Theorem \ref{thm: first moment inhomogeneous space} below.

\begin{theorem}[First Moments in Inhomogeneous Space] \label{thm: first moment inhomogeneous space}
The final solutions for the first moments of $S(t),I(t),$ and $R(t)$ in inhomogeneous space are:
\begin{align*}
m_{1}^{I}(t,x) &= \displaystyle \frac{1}{(2 \pi)^{d}} \int_{T^{d}} e^{\displaystyle [\kappa \hat{\mathcal{L}}(k) + (\beta - \gamma)]t} e^{\displaystyle -ikx} dk,\\
m_{1}^{S}(t,x) &= \displaystyle \frac{1}{(2 \pi)^d} \Bigg[ \int_{T^{d}} \rho_{0} e^{\displaystyle \kappa \hat{\mathcal{L}}(k) t} e^{\displaystyle -ikx}dk\\
&-\beta \displaystyle \int_{0}^{t} \sum\limits_{z \in \mathbb{Z}^{d}} \mathnormal{p} (t-s, 0, x-z) \left(\displaystyle \int_{T^{d}} e^{\displaystyle [\kappa \hat{\mathcal{L}}(k) + (\beta - \gamma)]s} e^{\displaystyle -ikz} dk \right)\Bigg] ds,\\
m_{1}^{R}(t,x) &= \displaystyle \frac{\gamma}{(2 \pi)^{d}} \int_{0}^{t} \sum\limits_{z \in \mathbb{Z}^{d}} \mathnormal{p} (t-s, 0, x-z)
\left(\displaystyle  \int_{T^{d}} e^{\displaystyle [\kappa \hat{\mathcal{L}}(k) + (\beta - \gamma)]s} e^{\displaystyle -ikz} dk \right)ds.
\end{align*}
\end{theorem}

Theorem \ref{thm: first moment inhomogeneous space} serves to elucidate the inherent mechanism of dynamic mobility and its influence on the susceptible, infected, and recovered groups. This establishes a fundamental groundwork for prospective investigations into the consequences of diverse transition probabilities within random motions on the dynamical system. Nonetheless, the equations presented in Theorem \ref{thm: first moment inhomogeneous space} are not immediately apparent, which renders the analysis of the asymptotic behavior of the three groups a challenging endeavor. To enhance the clarity of the solutions, we intend to reformat the equations into a matrix structure and subsequently derive solutions for the first moments. Let $U(t,x) = \begin{bmatrix}
S(t,x)\\
I(t,x)\\
R(t,x)
\end{bmatrix}$ and then $m_{1}^{U}(t,x) = \begin{bmatrix}
m_{1}^{S}(t,x)\\
m_{1}^{I}(t,x)\\
m_{1}^{R}(t,x)
\end{bmatrix}$. Rewriting the differential equations from Theorem \ref{thm: differential equations first moment} we have that
\begin{align}
\displaystyle \frac{\partial \hat{m}_{1}^{U}(t,k)}{\partial t} &= \hat{A}_{1} \hat{m}_{1}^{U}(t,k)\\
\hat{m}_{1}^{U}(t,k) &= e^{\displaystyle \hat{A}_{1} t} x_{0}
\end{align}
where $\hat{A}_{1}$ is the matrix consisting of the coefficients of the differential equations in Theorem \ref{thm: differential equations first moment}. The solution has 2 cases: $\beta=\gamma$ and $\beta \neq \gamma$.\\

When $\beta=\gamma$, the matrix format of the first moments of the susceptible, infected, and recovered groups is:\\
\begin{center}
${m}_{1}^{U}(t,x) = \begin{bmatrix}
\displaystyle \Big(\frac{1}{2 \pi} \Big)^{d} \int_{T^d} \Big(\rho_{0}  e^{\displaystyle \kappa \hat{a}(k)t} -\beta t e^{\displaystyle \kappa \hat{a}(k) t} \Big) e^{\displaystyle -ikx} dk \\
\displaystyle \Big(\frac{1}{2 \pi} \Big)^{d} \int_{T^d} \Big( e^{\displaystyle \kappa \hat{a}(k)t} \Big) \displaystyle e^{\displaystyle -ikx} dk \\
\displaystyle \Big(\frac{1}{2 \pi} \Big)^{d} \int_{T^d} \Big(\gamma t  e^{\displaystyle \kappa \hat{a}(k)t} \Big) e^{\displaystyle -ikx} dk \\ \end{bmatrix}$\\
\end{center}

When $\beta \neq \gamma$, the matrix format is: ${m}_{1}^{U}(t,x)=$ 

\begin{equation}\label{eq: matrix first moments s,i,r}
= \begin{bmatrix}
 \Big(\frac{1}{2 \pi} \Big)^{d} \int_{T^d} \Big[\rho_{0} e^{ \kappa \hat{a}(k)t} + \Big(\frac{\beta}{\beta-\gamma}\Big)  e^{ \kappa \hat{a}(k) t} - \Big(\frac{\beta}{\beta-\gamma}\Big)  e^{ (\kappa \hat{a}(k)+\beta-\gamma)t} \Big] e^{ -ikx} dk \\
\Big(\frac{1}{2 \pi} \Big)^{d} \int_{T^d} \Big(e^{ (\kappa \hat{a}(k)+\beta-\gamma)t} \Big) e^{ -ikx} dk \\
\Big(\frac{1}{2 \pi} \Big)^{d} \int_{T^d} \Big[\Big(\frac{\gamma}{\gamma-\beta}\Big) e^{\kappa \hat{a}(k)t} - \Big(\frac{\gamma}{\gamma-\beta}\Big) e^{ (\kappa \hat{a}(k)+\beta-\gamma)t} \Big] e^{ -ikx} dk 
\end{bmatrix}
\end{equation}

The equations presented in matrix \eqref{eq: matrix first moments s,i,r} offer a more straightforward means of examining the asymptotic trends of the susceptible, infected, and recovered cohorts. Importantly, it should be acknowledged that the formulations in matrix \eqref{eq: matrix first moments s,i,r} are, in fact, synonymous with the solutions outlined in Theorem \ref{thm: first moment inhomogeneous space}. To illustrate this equivalence, we will demonstrate the proof for the matrix-form equation pertaining to the recovered group, which corresponds to $m_{1}^{R}(t,x)$ as established in Theorem \ref{thm: first moment inhomogeneous space}. It is worth noting that analogous derivations can be applied to the other groups.

Recall from Theorem \ref{thm: first moment inhomogeneous space} that\\
\begin{center}
$m_{1}^{R}(t,x) = \displaystyle \frac{\gamma}{(2\pi)^d} \sum\limits_{z \in \mathbb{Z}^d} \int_{T^{d}} \int_{0}^{t} p(t-s,0,x-z) e^{\displaystyle (\kappa\hat{a}(k)+\beta-\gamma)s} e^{\displaystyle -ikz}ds dk$ 
\end{center}
and $p(t-s,0,x-z) = \displaystyle \frac{1}{(2\pi)^d} \int_{T^{d}} e^{\displaystyle ik_{1}(z-x)} e^{\displaystyle \kappa\hat{a}(k_{1})(t-s)}dk_{1}$.\\
Substituting $p(t-s,0,x-z)$ into $m_{1}^{R}(t,x)$, we have that
\begin{align*}
m_{1}^{R}(t,x) &= \displaystyle \frac{\gamma}{(2\pi)^d} \sum\limits_{z \in \mathbb{Z}^d} \int_{T^{d}} \int_{0}^{t} \displaystyle \frac{1}{(2\pi)^d} \int_{T^{d}} e^{\displaystyle -ik_{1}x} e^{\displaystyle \kappa\hat{a}(k_{1})t+(\beta-\gamma)s} dk_{1} ds dk,\\
&= \displaystyle \frac{\gamma}{(2\pi)^d} \sum\limits_{z \in \mathbb{Z}^d} \int_{T^{d}} \Big[\displaystyle \frac{1}{(2\pi)^d} \int_{T^{d}} e^{\displaystyle -ik_{1}x} e^{\displaystyle \kappa\hat{a}(k_{1})t}dk_{1}\Big] \Big[\Big(\displaystyle \frac{1}{\gamma- \beta}\Big) \Big(1-e^{\displaystyle (\beta-\gamma)t}\Big)\Big]dk,\\
&=\displaystyle \frac{\gamma}{(2\pi)^d} \int_{T^{d}} \sum\limits_{z \in \mathbb{Z}^d} p(t,0,x) \Big[\Big(\displaystyle \frac{1}{\gamma- \beta}\Big) \Big(1-e^{\displaystyle (\beta-\gamma)t}\Big)\Big]dk,\\
&=\displaystyle \frac{1}{(2\pi)^d} \int_{T^{d}} \Big(e^{\displaystyle \kappa\hat{a}(k)t} e^{\displaystyle -ikx} \Big)\Big[\Big(\displaystyle \frac{\gamma}{\gamma- \beta}\Big) \Big(1-e^{\displaystyle (\beta-\gamma)t}\Big)\Big]dk,\\
&=\displaystyle \frac{1}{(2\pi)^d} \int_{T^{d}} \Big[\Big(\displaystyle \frac{\gamma}{\gamma- \beta}\Big) \Big(e^{\displaystyle \kappa\hat{a}(k)t}-e^{\displaystyle (\kappa\hat{a}(k) + \beta-\gamma)t}\Big)\Big]e^{\displaystyle -ikx} dk,
\end{align*} which is equivalent to row $3$ in matrix \eqref{eq: matrix first moments s,i,r}. 

\subsection{Analyzing the First Moments}
Having obtained solutions for the first moments of the susceptible, infected, and recovered populations in both homogeneous and inhomogeneous spatial settings, we can now delve into the analysis of their asymptotic trends as time ($t$) approaches infinity in both scenarios. We will focus on the asymptotic behavior of the infected population. Without loss of generality, we consider the epidemic outbreak to begin at location $z=\Vec{0}$, $z\in Z^d$, where there is one infected individual at time $t=0$. 
In a homogeneous space, the asymptotic behavior of the infected population follows the classical SIR model dynamics. If $\beta < \gamma$, the infected population group will vanish as $t \to \infty$. If $\beta = \gamma$, the infected population stabilizes, and for $\beta > \gamma$, the infected population grows exponentially. This reflects the standard outcome of an infection spreading in a homogeneous population without mobility constraints.

In contrast, in an inhomogeneous space, where the mobility effect is represented by $\kappa \hat{a}(k) \leq 0$, $|\kappa \hat{a}|$ measures the strength of mobility across space. When $\theta = \kappa \hat{a}(k) + \beta - \gamma < 0$, implying that mobility outweighs the difference between infection and recovery rates, the infected population at location $x$ vanishes regardless of the relationship between $\beta$ and $\gamma$. This result demonstrates that human mobility can reduce the infection at a particular location, even when the infection rate equals the recovery rate, which differs from the classical SIR model where mobility is not considered.

For the scenario where $\theta = 0$, there are two cases: 
1. If $\kappa \hat{a}(k) = 0$, representing spatial isolation at location $x = 0$ (the origin of the epidemic), the infected population stabilizes at a finite constant $C_1$ at $x=0$, while it remains zero elsewhere. 
2. If $\alpha < 0$, meaning the mobility effect balances the difference between the infection and recovery rates ($|\alpha| = \beta - \gamma$), the infected population reaches a steady state.

In the case where $\theta > 0$ and $\kappa \hat{a}(k) = 0$, spatial isolation restricts the spread at location $x$. When $\beta > \gamma$, the infected population at $x = 0$ grows to infinity as $t \to \infty$, while at $x \neq 0$, the population vanishes. However, if $\alpha < 0$, indicating active mobility, and $\beta > \gamma + |\alpha|$, the infected population will grow exponentially, leading to widespread infection.

In homogeneous space, we have the same basic reproduction number $R_0=\frac{\beta}{\gamma}$ as in the traditional SIR model. If $R_0<1$, the infected population will eventually disappear. If $R_0>1$, an outbreak will occur. When $R_0=1$, the infected population will remain in a steady state, provided there is a sufficiently large population that may become infected.

In inhomogeneous space, we can define a similar reproduction number as in the classical SIR model (\citep{allen2008}):
$$
R_0^{m}=\frac{\kappa \hat{a}(k)+\beta}{\gamma}
$$
where $\kappa$ is the mobility rate, $\hat{a}(k)$ is the Fourier transform of the mobility kernel $a(z), \beta$ is the infection rate, and $\gamma$ is the recovery rate. If $R_0^{m}<1$, then $m_1^I(t, x) \rightarrow 0$, and the infected population will eventually vanish. If $R_0^{m}>1$,
(a) When $\alpha=0$: There will be an outbreak only at the original source location, while other locations remain safe with no infected individuals after a long time.
(b) When $\alpha<0$: Outbreaks will occur in all locations due to the influence of mobility. If $R_0^{m}=1$, the infected population will stay concentrated at the original location in a steady state.

    An intriguing question arises when the mobility kernel $a(z)$ is not symmetric, meaning $a(z) \neq$ $a(-z)$, while still satisfying $\sum_{z \in \mathbb{Z}^d} a(z)=0$ and $\sum_{z \neq 0} a(z)=1$. In this case, the Fourier transform $\hat{a}(k)$ becomes complex-valued:

$$
\hat{a}(k)=\sum_{z \in \mathbb{Z}^d} a(z) e^{i k \cdot z}=\operatorname{Re}[\hat{a}(k)]+i \operatorname{Im}[\hat{a}(k)],
$$
where the real and imaginary parts are given by $\operatorname{Re}[\hat{a}(k)]=\sum_z a(z) \cos (k \cdot z)$ and $\operatorname{Im}[\hat{a}(k)]=\sum_z a(z) \sin (k \cdot z)$, respectively.

The presence of the imaginary part in the exponent modifies the integral for the expectation of the infected population:

$$
m_1^I(t, x)=\frac{1}{(2 \pi)^d} \int_{T^d} e^{[\kappa \operatorname{Re}[\hat{a}(k)]+(\beta-\gamma)] t} e^{i[\kappa \operatorname{Im}[\hat{a}(k)] t-k \cdot x]} d k
$$

This expression splits into two components:

1. \textit{Amplitude Modulation}: The exponential of the real part $\kappa \operatorname{Re}[\hat{a}(k)]$ affects the magnitude of each Fourier mode, influencing the overall growth or decay of the infection.

2. \textit{Phase Shift}: The imaginary part $\kappa \operatorname{Im}[\hat{a}(k)]$ introduces a phase shift, causing the integral to represent a traveling wave or moving distribution, which reflects directional movement in the spread of the disease.

The integral now involves complex exponentials with non-trivial oscillatory behavior, making analytical solutions more challenging. The oscillations can lead to constructive or destructive interference, affecting the shape and amplitude of $m_1^I(t, x)$. We can conjecture that the infection spreads not only due to diffusion but also due to advection, leading to asymmetric spatial distributions.

Considering the asymmetry of $a(z)$, we conjecture that the modified reproduction number becomes
$$
R_0^m=R_0+\frac{\kappa}{\gamma} \max _{k \in T^d} \operatorname{Re}[\hat{a}(k)].
$$

When $\max _{k \in T^d} \operatorname{Re}[\hat{a}(k)]>0$, it is possible that $R_0^m>1>R_0$ or $R_0^m>R_0>1$. Both cases indicate that mobility enhances disease transmission compared to the classical SIR model without mobility. Conversely, when $\max _{k \in T^d} \operatorname{Re}[\hat{a}(k)]<0$, mobility hinders disease spread more than predicted by the classical model.

Investigating the nonsymmetric kernel in the nonlocal model presents an interesting direction for future research. We shall study this case rigorously in future. Understanding how asymmetry in movement patterns affects disease dynamics could provide valuable insights into controlling epidemics in populations with directional biases in mobility.

A succinct overview of the long-term asymptotic behavior of the equations in both homogeneous and inhomogeneous spatial settings is provided in Table \ref{tab:first_moment_asymptotic_horizontal_compact}.

\begin{table}[ht]
\resizebox{\textwidth}{!}{
\begin{tabular}{|c|c|c|c|}
\hline
\textbf{Space Type} & \textbf{Condition} & \textbf{Asymptotic Behavior} & \textbf{Notes} \\ 
\hline
\multirow{3}{*}{Homogeneous} & $\beta < \gamma$ or $R_0<1$ & $m_1^I(t,x) \to 0$ & Infected population goes to zero \\
& $\beta = \gamma$ or $R_0=1$& $m_1^I(t,x) = C_0 \in R $ & Infected population reaches steady state \\
& $\beta > \gamma$ or $R_0>1$ & $m_1^I(t,x) \to e^{(\beta-\gamma)t} \to \infty$ & Infected population grows exponentially \\
\hline
\multirow{3}{*}{Inhomogeneous} & $\theta < 0$ or $R_0^m<1$ & $m_1^I(t,x) \to 0$ & Infected population goes to zero \\
& $\theta = 0$ or $R_0^m=1$ & $m_1^I(t,x) \rightarrow \delta_0(x)$ & Dirac delta function at $x = 0$ \\
&$\alpha = 0$, $\theta > 0$ or $R_0^m>1$ & $\infty$ for $x=0$, $0$ for $x \neq 0$ & Infection only at $x=0$ \\
& $\alpha < 0$, $\theta > 0$ or $R_0^m>1$  & $m_I^1(t,x)\rightarrow \infty$ & Infected population goes to infinity \\
\hline
\end{tabular}
}
\caption{Denote $\theta = \kappa \hat{a}(k) + \beta - \gamma$, $\alpha = \kappa \hat{a}(k)$. $R_0=\frac{\beta}{\gamma}$. $R_0^m=\frac{\kappa\hat{a}(k)+\beta}{\gamma}$. Asymptotic Behavior of $m_1^I(t,x)$ in Homogeneous and Inhomogeneous Space.}
\label{tab:first_moment_asymptotic_horizontal_compact}
\end{table}

This analysis highlights the significant role that mobility plays in disease spread, providing mathematical insights into the interactions between isolation, mobility, and infection dynamics. A key factor is the Fourier transform of the mobility distribution, which allows for the classification of mobility kernels into several groups. By analyzing the Fourier transform, we can predict whether a specific location is at risk for an outbreak. Estimating the transition probability kernel involves using real-world mobility data, such as human movement patterns from GPS data, transportation networks, or census data, and fitting these to mathematical models. Other methods, for example, maximum likelihood estimation (MLE) or Bayesian inference can then be applied to estimate the parameters of the mobility kernel. Mobility can either suppress or amplify infection depending on the balance between infection, recovery rates, and mobility, with significant implications for controlling epidemics in real-world scenarios.

\section{The Second Moments}

In this section, we will solve for and study the second moments of $S(t,x), I(t,x), R(t,x)$ and the corresponding intersecting groups. We denote 
\begin{align*}
&v = \Vert y - x \Vert, \ m_{2}^{S}(t,x,y) = E[S(t,x)S(t,y)], \ m_{2}^{SI}(t,x,y) = E[S(t,x)I(t,y)],\\
&\mathcal{L}f(t,x) = \sum\limits_{z \neq 0} a(z)(f(t,x+z)-f(t,x)),\\
&\mathcal{L}_{x}f(t,x,y) = \sum\limits_{z \neq 0} a(z)(f(t,x+z,y)-f(t,x,y)),\\
&\mathcal{L}_{y}f(t,x,y) = \sum\limits_{z \neq 0} a(z)(f(t,x,y+z)-f(t,x,y)),\\
&\mathcal{L}_{S_{x}}m_{2}^{SI}(t,x,x) = \sum\limits_{z \neq 0} a(z)E[I(t,x)S(t,x+z)-I(t,x)S(t,x)],\\
&\mathcal{L}_{I_{x}} m_{2}^{SI}(t,x,x) = \sum\limits_{z \neq 0} a(z)E[S(t,x)I(t,x+z)-S(t,x)I(t,x)].
\end{align*}

\subsection{Deriving the Differential Equations}
First, we need to derive the differential equations of the second moments for the $S(t,x), I(t,x), R(t,x), S(t,x)I(t,x)$, and $R(t,x)I(t,x)$ groups.

\begin{theorem} \label{thm: diff eq for sec moment}
The differential equations for the second moment among the susceptible, infected, recovered groups are, when $x=y$,
\begin{small}
\begin{align*}
&\displaystyle \frac{\partial m_{2}^{S}(t,x,x)}{\partial t}= 2\kappa \mathcal{L}_{x}m_{2}^{S}(t,x,x) + \kappa \mathcal{L}_{x} m_{1}^{S}(t,x) - 2 \beta m_{2}^{SI}(t,x,x) + 2 \kappa m_{1}^{S}(t,x)+ \beta m_{1}^{I}(t,x).\\
&\displaystyle \frac{\partial m_{2}^{I}(t,x,x)}{\partial t} =  2 \kappa \mathcal{L}_{x} m_{2}^{I}(t,x,x) + \kappa \mathcal{L} m_{1}^{I}(t,x) + 2(\beta - \gamma) m_{2}^{I}(t,x,x) + 2 \kappa m_{1}^{I}(t,x)+ (\beta + \gamma) m_{1}^{I}(t,x).\\
&\displaystyle \frac{\partial m_{2}^{R}(t,x,x)}{\partial t} = 2 \kappa \mathcal{L}_{x} m_{2}^{R}(t,x,x) + \kappa \mathcal{L} m_{1}^{R}(t,x) + 2 \gamma m_{2}^{RI}(t,x,x) + 2 \kappa m_{1}^{R}(t,x)+ \gamma m_{1}^{I}(t,x).\\
&\displaystyle\frac{\partial m_{2}^{SI}(t,x,x)}{\partial t}=  \kappa \mathcal{L}_{S_{x}}m_{2}^{SI}(t,x,x) + \kappa  \mathcal{L}_{I_{x}}m_{2}^{SI}(t,x,x) + (\beta - \gamma)m_{2}^{SI}(t,x,x)- \beta m_{2}^{I}(t,x,x) - \beta m_{1}^{I}(t,x).\\
&\displaystyle\frac{\partial m_{2}^{RI}(t,x,x)}{\partial t} =  \kappa \mathcal{L}_{R_{x}}m_{2}^{RI}(t,x,x) + \kappa  \mathcal{L}_{I_{x}}m_{2}^{RI}(t,x,x) + (\beta - \gamma)m_{2}^{RI}(t,x,x)+ \gamma m_{2}^{I}(t,x,x) - \gamma m_{1}^{I}(t,x).
\end{align*}
\end{small}
When $x \neq y$,
\begin{small}
\begin{align*}
\displaystyle \frac{\partial m_{2}^{S}(t,v)}{\partial t} &= \kappa \mathcal{L}_{x} m_{2}^{S}(t,v) + \kappa \mathcal{L}_{y} m_{2}^{S}(t,v) - 2 \beta m_{2}^{SI}(t,v) - \kappa a(v) m_{1}^{S}(t,x)-\kappa a(v) m_{1}^{S}(t,y).\\
\displaystyle \frac{\partial m_{2}^{I}(t,v)}{\partial t} &= \kappa \mathcal{L}_{x} m_{2}^{I}(t,v) + \kappa \mathcal{L}_{y} m_{2}^{I}(t,v) + 2(\beta- \gamma) m_{2}^{I}(t,v) - \kappa a(v) m_{1}^{I}(t,x)- \kappa a(v) m_{1}^{I}(t,y).\\
\displaystyle \frac{\partial m_{2}^{R}(t,v)}{\partial t} &= \kappa \mathcal{L}_{x} m_{2}^{R}(t,v) + \kappa \mathcal{L}_{y} m_{2}^{R}(t,v) + 2 \gamma m_{2}^{RI}(t,v) -\kappa a(v)m_{1}^{R}(t,x)- \kappa a(v)m_{1}^{R}(t,y).\\
\displaystyle \frac{\partial m_{2}^{SI}(t,v)}{\partial t} &= \kappa \mathcal{L}_{x}m_{2}^{SI}(t,v) + \kappa \mathcal{L}_{y}m_{2}^{SI}(t,v) + (\beta-\gamma)m_{2}^{SI}(t,v) - \beta m_{2}^{I}(t,v). \\
\displaystyle \frac{\partial m_{2}^{RI}(t,v)}{\partial t} &= \kappa \mathcal{L}_{x}m_{2}^{RI}(t,v) + \kappa \mathcal{L}_{y}m_{2}^{RI}(t,v) + (\beta-\gamma)m_{2}^{RI}(t,v) + \gamma m_{2}^{I}(t,v).
\end{align*}
\end{small}
\end{theorem}

\begin{proof} For the susceptible group, when deriving the differential equations for the second moment, we are going to use a similar method to the one used for the first moment - the Kolmogorov forward Equations. There are 2 cases: when the locations are equivalent and $x=y$, and when the locations are different and $x\neq y$:\\
Case 1: $S(t+dt,x)$ when $x = y$, then $m_{2}(t+dt,x,y) = E[S^{2}(t+dt, x, x)]$\\
For the second moment, when $x=y$, we have that\\
$E[S^{2}(t+dt,x)] = E[(S(t,x)+ \xi(dt))^{2}]$, where
\vspace{-3mm}
$$\xi (dt)= 
 \begin{cases}
        1 \hspace{.25cm} w.p \hspace{.25cm}\indent \sum\limits_{z \neq 0}S(t, x+z) \kappa a(z) dt \hspace{.25cm} \text{for event (ii)}\\
        -1 \hspace{.25cm} w.p \hspace{.25cm} \sum\limits_{z \neq 0}S(t, x) \kappa a(z) dt + I(t,x) \beta dt \hspace{.25cm} \text{for events (i),(vii)}\\
        0 \hspace{.25cm} \text{otherwise}
    \end{cases}$$
The case that $\xi(dt)=1$ is the event that a particle at location $x+z$ in the susceptible group moves to location $x$, meaning location $x$ gains a particle, and thus the event has probability $\sum\limits_{z \neq 0}S(t, x+z) \kappa a(z) dt$. The case that $\xi(dt)=-1$ is the event that either a particle at location $x$ moves to location $x+z$ (meaning that location $x$ loses a particle, which has probability $\sum\limits_{z \neq 0}S(t, x) \kappa a(z) dt$), or a particle at location $x$ in the susceptible group becomes infected (which has probability $I(t,x) \beta dt$). The case that $\xi(dt)=0$ is the event that there is no particle moving to or away from location $x$ in the susceptible group, so it has probability $1- \sum\limits_{z \neq 0}S(t, x+z) \kappa a(z) dt - \sum\limits_{z \neq 0}S(t, x) \kappa a(z) dt - I(t,x) \beta dt$. Using the Kolmogorov Forward Equations and we have that\\

$E[S^{2}(t+dt,x)] = E[S^{2}(t,x)] + 2E\big[S(t,x)(1)(\sum\limits_{z \neq 0}S(t, x+z) \kappa a(z) dt) +\\
S(t,x)(-1)(\sum\limits_{z \neq 0}S(t, x) \kappa a(z) dt + I(t,x) \beta dt) + S(t,x)(0)[1 - \sum\limits_{z \neq 0}S(t, x+z) \kappa a(z) dt - (\sum\limits_{z \neq 0}S(t, x) \kappa a(z) dt + I(t,x) \beta dt)]\big] + E\big[(1)^{2}(\sum\limits_{z \neq 0}S(t, x+z) \kappa a(z) dt) + (-1)^{2}(\sum\limits_{z \neq 0}S(t, x) \cdot\\
\kappa a(z) dt + I(t,x) \beta dt) + (0)^{2} [1 - \sum\limits_{z \neq 0}S(t, x+z) \kappa a(z) dt - (\sum\limits_{z \neq 0}S(t, x) \kappa a(z) dt + I(t,x) \beta dt)]\big]$.

Hence
\begin{small}
\begin{equation*}
    \displaystyle \frac{\partial m_{2}^{S}(t,x,x)}{\partial t}= 2\kappa \mathcal{L}_{x}m_{2}^{S}(t,x,x) + \kappa \mathcal{L} m_{1}^{S}(t,x) - 2 \beta m_{2}^{SI}(t,x,x) + 2 \kappa m_{1}^{S}(t,x) + \beta m_{1}^{I}(t,x)
\end{equation*}
\end{small}
Case 2: $S(t+dt,x)$ when $x \neq y$, then $m_{2}(t+dt,x,y) = E[S(t+dt,x)S(t+dt,y)]$\\
For the second moment when $x \neq y$ we have that,
\begin{center}
$E[S(t+dt,x)S(t+dt,y)] = E[(S(t,x)+ \xi(dt,x))(S(t,y)+ \xi(dt,y))]$
\end{center}
Recall that during $(t, t+dt)$ only one event can happen, either a particle can move or it can jump states, therefore there are several combinations for $x$ and $y$. For example, $(x=1,y=-1)$, $(x=1,y=0)$, so on and so forth. The probabilities for the various combinations of $\xi(dt,x)$ and $\xi(dt,y)$ are:
\begin{itemize}
\item The event that a particle in the susceptible group goes from $y$ to $x$:
\begin{center}
$P(\xi(dt,x)= 1, \xi(dt,y)= -1)= \kappa a(x-y) S(t,y)dt$
\end{center}
\item The event that a particle in the susceptible group goes from $x+z$ to $x$ but not to $y$:
\begin{center}
$P(\xi(dt,x)= 1, \xi(dt,y)= 0)= \kappa \sum\limits_{z \neq 0} a(z) S(t,x+z)dt - \kappa a(x-y) S(t,y)dt$
\end{center}
\item The event that a particle in the susceptible group goes from $x$ to $y$:
\begin{center}
$P(\xi(dt,x)= -1, \xi(dt,y)= 1)= \kappa a(x-y) S(t,x)dt$
\end{center}
\item The event that a particle in the susceptible group goes from $x$ to $x+z$ but not to $y$ or a particle at location $x$ transitions $S \rightarrow I$:
\begin{center}
$P(\xi(dt,x)= -1, \xi(dt,y)= 0)= \kappa \sum\limits_{z \neq 0} a(z) S(t,x)dt - \kappa a(x-y) S(t,x)dt + \beta I(t,x)dt$
\end{center}
\item The event that a particle in the susceptible group goes from $y+z$ to $y$ but not to $x$:
\begin{center}
$P(\xi(dt,x)= 0, \xi(dt,y)= 1)= \kappa \sum\limits_{z \neq 0} a(z) S(t,y+z)dt - \kappa a(x-y) S(t,x)dt$
\end{center}
\item The event that a particle in the susceptible group goes from $y$ to $y+z$ but not to $x$ or a particle at location $y$ transitions $S \rightarrow I$:
\begin{center}
$P(\xi(dt,x)= 0, \xi(dt,y)= -1)= \kappa \sum\limits_{z \neq 0} a(-z) S(t,y)dt - \kappa a(x-y) S(t,y)dt + \beta I(t,y)dt$ 
\end{center}
\item The event that no particle moves in the susceptible group:
\begin{align*}
P(\xi(dt,x)= 0, \xi(dt,y)= 0) &= 1 - \kappa \sum\limits_{z \neq 0} a(z)S(t,y+z)dt - \kappa \sum\limits_{z \neq 0} a(z)S(t,y)dt\\
&- \kappa \sum\limits_{z \neq 0} a(z)S(t,x+z)dt- \kappa \sum\limits_{z \neq 0} a(z)S(t,x)dt\\
&+ \kappa a(x-y) S(t,y)dt + \kappa a(x-y) S(t,x)dt\\
&- \beta I(t,x)dt - \beta I(t,y)dt
\end{align*}
\end{itemize}

Using the Kolmogorov Forward Equations, multiply out the quantities and cancel terms, re-write sums in terms of our Laplace operators defined in the beginning of the chapter, distribute the expectation and divide both sides by $dt$, and let $v = \Vert y-x \Vert$ we get that
\begin{small}
\begin{align*}
\displaystyle \frac{\partial m_{2}^{S}(t,v)}{\partial t} &= \kappa \mathcal{L}_{x} m_{2}^{S}(t,v) + \kappa \mathcal{L}_{y} m_{2}^{S}(t,v) - 2 \beta m_{2}^{SI}(t,v)- \kappa a(v) m_{1}^{S}(t,x) -\kappa a(v) m_{1}^{S}(t,y)
\end{align*}
\end{small}
\end{proof}

The proof for deriving the differential equations for the infected, susceptible-infected, recovered, and recovered-infected groups follows in a similar process as the susceptible group with different probabilities for $\xi(dt)$ depending on the group. 

\subsection{Solving for the Second Moments}

In epidemic models, the most important group is the infected population. Here we will only list the equations for the infected population. However, solving solve the differential equations of second moments for the $S(t,x,y), R(t,x,y), S(t,x)I(t,y)$, and $R(t,x)I(t,y)$ groups follows in a similar manner. We will utilize regular ODE methods and apply Theorem \ref{thm: general solution} to solve for the second moment of the infected population. Each of the second moments has 2 cases: when the locations $x=y$ and when the locations $x \neq y$ and each case has 2 subcases: homogeneous space and inhomogeneous space.

\begin{theorem}\label{thm: 2nd moment}
The second moments for the infected groups when $x=y$ and $x \neq y$ are:
\begin{align*}
\text{Homogeneous space}:\\
m_{2}^{I}(t,x,x) &= \rho_{0} e^{\displaystyle 2 (\beta-\gamma)t} + \Big(\displaystyle \frac{\beta + \gamma + 2\kappa}{-\beta + \gamma}\Big) \Big[ e^{\displaystyle (\beta-\gamma)t} - e^{\displaystyle 2(\beta-\gamma)t} \Big].\\
\text{Inhomogeneous space}:\\
m_{2}^{I}(t,x,x) &= \displaystyle \frac{1}{(2 \pi)^d} \int_{T^d} \rho_{0} e^{\displaystyle [2\kappa \hat{\mathcal{L}}_{x}(k) + 2(\beta-\gamma)]t} e^{\displaystyle -ikx}dk\\
&+ e^{\displaystyle 2(\beta-\gamma)t} \displaystyle\int_{0}^{t} \sum\limits_{z \in \mathbb{Z}^d} p(t-s,0,x-z)\Big[\kappa \mathcal{L}m_{1}^{I}(s,z)\\
&+(\beta+\gamma+2\kappa)m_{1}^{I}(s,z)\Big]\Big(e^{\displaystyle 2(\beta-\gamma)s}\Big)ds.\\
\text{Homogeneous space}:\\
m_{2}^{I}(t,v) &=  \rho_{0} e^{\displaystyle 2 (\beta-\gamma)t} + \Big(\displaystyle \frac{2 \kappa a(v)}{(\beta-\gamma)}\Big) \Big[ e^{\displaystyle (\beta-\gamma)t} - e^{\displaystyle 2 (\beta - \gamma)t} \Big].\\
\text{Inhomogeneous space}:\\
m_{2}^{I}(t,v) &=\displaystyle \frac{1}{(2\pi)^d} \int_{T^d} \rho_{0} e^{\displaystyle [\kappa \hat{\mathcal{L}}_{x}(k) + \kappa \hat{\mathcal{L}}_{y}(k) + 2 (\beta- \gamma)]t} e^{\displaystyle -ikv}dk\\
&+e^{\displaystyle 2(\beta-\gamma)t} \displaystyle\int_{0}^{t} \sum\limits_{z \in \mathbb{Z}^d} p(t-s,0,x-z)\Big[-2 \kappa a(v) m_{1}^{I}(s,z)\Big]\\
&\cdot \Big(e^{\displaystyle 2(\beta-\gamma)s}\Big)ds.
\end{align*}
\end{theorem}

\begin{proof}
Recall the differential equations for the infected group from Theorem \ref{thm: diff eq for sec moment} and for each equation $x=y$ and $x \neq y$ there are 2 cases: homogeneous space and inhomogeneous space.

Case 1a: Homogeneous space when $x=y$ (then the spaces $x$ and $x+z$ are equivalent and $\mathcal{L}_{x} m_{2}^{I}(t,x,x)=0= \mathcal{L} m_{1}^{I}(t,x)$).\\

From Theorem \ref{thm: diff eq for sec moment} we have $\displaystyle \frac{\partial m_{2}^{I}(t,x,x)}{\partial t} =  2(\beta - \gamma) m_{2}^{I}(t,x,x) + (\beta + \gamma + 2\kappa) e^{\displaystyle (\beta-\gamma)t}$ with initial condition $m_{2}^{I}(0,x,x) = \rho_{0}>0$. Thus,
\begin{center}
$m_{2}^{I}(t,x,x) = \rho_{0} e^{\displaystyle 2 (\beta-\gamma)t} + \Big(\displaystyle \frac{\beta + \gamma + 2\kappa}{-\beta + \gamma}\Big) \Big[ e^{\displaystyle (\beta-\gamma)t} - e^{\displaystyle 2(\beta-\gamma)t} \Big].$
\end{center}

Case 1b: Homogeneous space when $x \neq y$ (then $\mathcal{L}_{x}m_{2}^{I}(t,v) = 0$ and $m_{1}^{I}(t,x)=m_{1}^{I}(t,y)= e^{\displaystyle (\beta-\gamma)t}$):
\begin{align*}
\displaystyle \frac{\partial m_{2}^{I}(t,v)}{\partial t} &= 2(\beta - \gamma) m_{2}^{I}(t,v) - 2 \kappa a(v) e^{\displaystyle (\beta-\gamma)t},\\
m_{2}^{I}(0,x,x) &= \rho_{0}>0.
\end{align*}
When we solve this ODE, we get
\begin{center}
$m_{2}^{I}(t,v) = \rho_{0} e^{\displaystyle 2 (\beta-\gamma)t} + \Big(\displaystyle \frac{2 \kappa a(v)}{(\beta-\gamma)}\Big) \Big[ e^{\displaystyle (\beta-\gamma)t} - e^{\displaystyle 2 (\beta - \gamma)t} \Big]$.
\end{center}

Case 2a: In-homogeneous space when $x=y$ (then the spaces $x$ and $x+z$ are not equivalent and $\mathcal{L}_{x} m_{2}^{I}(t,x,x) \neq 0$):
\begin{align*}
\displaystyle \frac{\partial m_{2}^{I}(t,x,x)}{\partial t} &=  2 \kappa \mathcal{L}_{x} m_{2}^{I}(t,x,x) + \kappa \mathcal{L} m_{1}^{I}(t,x) + 2(\beta - \gamma) m_{2}^{I}(t,x,x)\\
& + (\beta + \gamma +2\kappa) m_{1}^{I}(t,x),\\ 
m_{2}^{I}(0,x,x) &= \rho_{0}>0.
\end{align*}
Utilizing the general solution for inhomogeneous equations from Theorem \ref{thm: general solution}, we have that
\begin{align*}
m_{2}^{I}(t,x,x) &= \displaystyle \frac{1}{(2 \pi)^d} \int_{T^d} \rho_{0} e^{\displaystyle [2\kappa \hat{\mathcal{L}}_{x}(k) + 2(\beta-\gamma)]t} e^{\displaystyle -ikx}dk\\
&+ e^{\displaystyle 2(\beta-\gamma)t} \displaystyle\int_{0}^{t} \sum\limits_{z \in \mathbb{Z}^d} p(t-s,0,x-z)\Big[\kappa \mathcal{L}m_{1}^{I}(s,z)+(\beta+\gamma+2\kappa)m_{1}^{I}(s,z)\Big]\\\
&\cdot \Big(e^{\displaystyle 2(\beta-\gamma)s}\Big)ds.
\end{align*}

Case 2b: Inhomogeneous space when $x \neq y$ (then $\mathcal{L}_{x}m_{2}^{I}(t,v) \neq 0$):
\begin{align*}
\displaystyle \frac{\partial m_{2}^{I}(t,v)}{\partial t} &= \kappa \mathcal{L}_{x} m_{2}^{I}(t,v) + \kappa \mathcal{L}_{y} m_{2}^{I}(t,v) + 2(\beta- \gamma) m_{2}^{I}(t,v)\\
&- \kappa a(v) m_{1}^{I}(t,x) - \kappa a(v) m_{1}^{I}(t,y),\\
m_{2}^{I}(0,x,x) &= \rho_{0}>0.
\end{align*}
When we apply Theorem \ref{thm: general solution} to the differential equation, we have that
\begin{align*}
m_{2}^{I}(t,v) &= \displaystyle \frac{1}{(2\pi)^d} \int_{T^d} \rho_{0} e^{\displaystyle [\kappa \hat{\mathcal{L}}_{x}(k) + \kappa \hat{\mathcal{L}}_{y}(k) + 2 (\beta- \gamma)]t} e^{\displaystyle -ikv}dk\\
&+ e^{\displaystyle 2(\beta- \gamma) t} \displaystyle \int_{0}^{t} \sum\limits_{z \in \mathbb{Z}^d} p(t-s,0,x-z) \Big[-2 \kappa a(v)m_{1}^{I}(s,z)\Big]\Big(e^{\displaystyle 2(\beta-\gamma)s}\Big)ds.
\end{align*}
\end{proof}

Solving the differential equations for the second moments of the susceptible, recovered, susceptible-infected, and recovered-infected groups follows the same procedure as solving for the second moment of the infected group.

Utilizing the same matrix method used for the first moments, except we have multiple cases for the second moments, we can solve for the second moments of the infected group explicitly.

When $\beta= \gamma$ and $x=y$,\\
$m_{2}^{I}(t,x,x) = \Big(\frac{1}{2\pi}\Big)^{d} \int_{T^d} \Big\{ e^{ 2 \kappa \hat{a}(k)t} - \Big(\frac{\kappa\hat{a}(k)+2\kappa+\beta+\gamma}{\kappa\hat{a}(k)}\Big) \Big[e^{\displaystyle \kappa \hat{a}(k)t} - e^{\displaystyle 2\kappa \hat{a}(k)t} \Big] \Big\} e^{\displaystyle -ikx}dk$.\\

When $\beta \neq \gamma$ and $x=y$,\\
$\hat{m}_{2}^{I}(t,x,k) = 
e^{\displaystyle 2(\kappa\hat{a}(k)t+\beta+\gamma)t} - \Big(\frac{\kappa\hat{a}(k)+2\kappa+\beta+\gamma}{\kappa\hat{a}(k)+2\beta+2\gamma}\Big)\Big[e^{\displaystyle \kappa\hat{a}(k)t} - e^{\displaystyle 2(\kappa\hat{a}(k)+\beta+\gamma)t}\Big]$,\\

$m_{2}^{I}(t,x,x)= \displaystyle \Big(\frac{1}{2\pi}\Big)^{d} \int_{T^d} \hat{m}_{2}^{I}(t,x,k) e^{\displaystyle -ikx} dk$.

When $\beta=\gamma$ and $x \neq y$,\\
$\hat{m}_{2}^{I}(t,k)= 
e^{\displaystyle 2 \kappa \hat{a}(k)t} + 2 \Big[e^{\displaystyle \kappa \hat{a}(k)t} - e^{\displaystyle 2\kappa \hat{a}(k)t} \Big]$,\\

$m_{2}^{I}(t,v)= \displaystyle \Big(\frac{1}{2\pi}\Big)^{d} \int_{T^d} \hat{m}_{2}^{I}(t,k) e^{\displaystyle -ikv} dk$.\\

When $\beta \neq \gamma$ and $x \neq y$,\\
$\hat{m}_{2}^{I}(t,k) = 
e^{\displaystyle 2(\kappa\hat{a}(k)t+\beta+\gamma)t} + \Big(\frac{2 \kappa\hat{a}(k)}{\kappa\hat{a}(k)+2\beta+2\gamma}\Big)\Big[e^{\displaystyle \kappa\hat{a}(k)t} - e^{\displaystyle 2(\kappa\hat{a}(k)+\beta+\gamma)t}\Big]$.\\

\subsection{Analyzing the Second Moments}

We can analyze the asymptotic behavior as $t \rightarrow \infty$ for both the cases (homogeneous space and inhomogeneos space), each with two subcases when $x=y$ and when $x \neq y$.

The behavior of the second moments of the susceptible, infected, recovered, susceptible-infected, recovered-infected groups in the homogeneous space as $t \longrightarrow \infty$ can be broken up in 3 cases: $\beta<\gamma, \beta=\gamma, \beta>\gamma$. The asymptotic behavior of  $m_{2}^{I}(t,x,x)$ and $m_{2}^{I}(t,v)$, where $v = \Vert y - x \Vert$ in homogeneous space as $t \longrightarrow \infty$ is summarized in the Table \ref{tab: asymptotic sec moment homogenous} below.

\begin{table}[ht]
\begin{tabular}{ | c | c | c | }
\hline
 As $t \rightarrow \infty$ & $m_{2}^{I}(t,x,x)$ & $m_{2}^{I}(t,v)$\\
\hline
$\beta > \gamma$ & $ \infty$ & $ \infty$\\ 
\hline
$\beta < \gamma$ & $ 0$ & $ 0$ \\
\hline
$\beta = \gamma$ & $ \infty$ & $ 0$ \\
\hline
\end{tabular}
\caption{Asymptotic Behavior of the Second Moments.}
\label{tab: asymptotic sec moment homogenous}
\end{table}

The asymptotic behavior of the second moments of the infected group in inhomogeneous space as $t \longrightarrow \infty$ is summarized in Table \ref{tab:asymptotic sec moment inhomogeneous} below. Let $\alpha = \kappa \hat{a}(k)$, $\theta = \kappa\hat{a}(k)+\beta-\gamma$, $\mu = \kappa\hat{a}(k)+\beta+\gamma$, $D_{1}= \displaystyle \frac{1}{(2\pi)^d}\int_{T^d}e^{\displaystyle -ikx} dk = \delta_0(x)$, $D_{2}= \displaystyle \frac{1}{(2\pi)^d}\int_{T^d} \Big(\displaystyle \frac{\kappa+\beta+\gamma}{2\beta+2\gamma}\Big) e^{\displaystyle -ikx}dk = \Big(\displaystyle \frac{\kappa+\beta+\gamma}{2\beta+2\gamma}\Big) \delta_0(x)$, and $D_{3}= \displaystyle \frac{1}{(2\pi)^d}\int_{T^d}e^{\displaystyle -ikv} dk= \delta_0(v)$.

\begin{table}[ht]
\begin{tabular}{ | c | c | c |  }
\hline
 As $t \rightarrow \infty$ & $m_{2}^{I}(t,x,x)$ & $m_{2}^{I}(t,v) $\\
\hline
$\beta - \gamma=0$, $\alpha < 0$, $\theta < 0 $ &  $ 0$ & $ 0$\\ 
\hline
$\beta - \gamma=0$, $\alpha = 0$, $\theta = 0$ & $\delta_0(x)$ & $\delta_0(v)$ \\ 
\hline 
$\beta + \gamma < 0$, $\alpha < 0$, $\mu < 0$ & $ 0$ & $0$ \\ 
\hline
$\beta + \gamma < 0$, $\alpha = 0$, $\mu<0$ & $ \Big( \frac{\kappa+\beta+\gamma}{2\beta+2\gamma}\Big) \delta_0(x) $ & $ 0$ \\
\hline 
$\beta + \gamma > 0$, $\alpha < 0$, $\mu <0$ & $ 0$ & $ 0$ \\
\hline
$\beta + \gamma > 0$, $\alpha \leq 0$, $\mu>0$ & $ \infty$ & $\infty$ \\
\hline
\end{tabular}
\caption{Asymptotic Behavior of the Second Moments in Inhomogeneous Space. }
\label{tab:asymptotic sec moment inhomogeneous}
\end{table}

Notice that when $\beta + \gamma >0, \alpha < 0$ and $\mu = \kappa\hat{a}(k)+\beta+\gamma < 0$, we have the infection rate plus the recovery rate ($\beta+\gamma$) is positive but smaller than the mobility effect $\kappa\hat{a}(k)$, meaning the mobility effect is stronger. The result is that the second moment of the infected population converges to  $0$ as time $t$ approaches infinity, meaning if the first moment also converges to  $0$ then the variance is $0$. Another notable event is when $\beta=\gamma$, meaning the infection rate is equal to the recovery rate. For the case when $\beta=\gamma$ and the mobility effect $\kappa\hat{a}(k) < 0$, the second moment of the infected population at location $x$ converges to  to $0$ as time $t$ approaches infinity and the first moment also converges to $0$, which means there is no variance and the location no longer has an infected population. The event where $\beta=\gamma$ and the mobility effect $\kappa\hat{a}(k) = 0$, we have that the second moment of the infected population at location $x$ converges to  a finite constant $C_{3}$ as $t$ approaches to infinity. This means that the second moment of the infected population converges to a steady state and the first moment of the infected group converges to  a constant, therefore the variance is fixed and the Central Limit Theorem applies and the infected population follows a normal distribution.

\section{Intermittency Analysis}

We now turn to the study of the \textit{intermittency effect} in our SIR model. Intermittency occurs when a system shows high-intensity activity in small, isolated regions, known as \textit{intermittent islands}, which carry most of the system’s total mass (\citep{carmona1994parabolic,8,konig2016parabolic}). Molchanov defines intermittency as the tendency for the ratio of second moments to the square of first moments to approach infinity over time in specific areas (\citep{8}). \citep{konig2016parabolic} further describes intermittency as large fluctuations in a field, leading to clusters with high concentrations of particles, separated by regions of low activity.

In physics, intermittency is commonly observed in \textit{turbulence}, where energy is concentrated in small regions of space and time, resulting in intense, sporadic bursts of activity. A similar phenomenon is seen in \textit{optics}, where intermittent light flashes, known as \textit{rogue waves}, occur in nonlinear media. In \textit{plasma physics}, concentrated regions of energy and particle density intermittently form within turbulent plasmas.

In our SIR model, intermittency within the infected population suggests the formation of clusters of infected individuals rather than uniform mixing. To quantify this, we define $m_2^I(t,x,y) = E[I(t,x)I(t,y)]$, representing the second moment of the infected population. The following criteria indicate clustering (intermittency):
\begin{itemize}
    \item For $x = y$, $\lim_{t \rightarrow \infty} \frac{m_2^I(t,x,x)}{(m_1^I(t,x))^2} = \infty$,  it means that the infected population becomes highly concentrated at specific, isolated points, indicating clusters at single locations.
    \item For $x \neq y$, $\lim_{t \rightarrow \infty} \frac{m_2^I(t,x,y)}{m_1^I(t,x)m_1^I(t,y)} = \infty$, it implies correlations between infected populations across distinct locations, indicating clustering over multiple, spread-out areas.
\end{itemize}

We consider four scenarios: homogeneous and inhomogeneous spaces for both $x = y$ and $x \neq y$. Intermittency in these cases reveals whether infection clusters develop and how infected individuals spatially distribute over time.

\textbf{When $x=y$ in Homogeneous Space:}\\
Substituting the equations from Theorem \ref{thm:first moment homogeneous space} and Theorem \ref{thm: 2nd moment} into $\displaystyle \frac{m_{2}^{I}(t,x,x)}{m_{1}^{I}(t,x)^2}$ and simplifying the equation we get
\begin{center}
$\displaystyle \frac{m_{2}^{I}(t,x,x)}{(m_{1}^{I}(t,x))^{2}} = \rho_{0} + \Big(\displaystyle \frac{\beta + \gamma + 2\kappa}{-\beta + \gamma}\Big) \Big[ e^{\displaystyle -(\beta-\gamma)t} - 1 \Big]$.
\end{center}

If $\beta < \gamma$, $\lim\limits_{t \rightarrow \infty} e^{\displaystyle -(\beta-\gamma)t} - 1 \rightarrow \infty$, thus $\lim\limits_{t \rightarrow \infty} \displaystyle \frac{m_{2}^{I}(t,x,x)}{(m_{1}^{I}(t,x))^{2}} \rightarrow \infty$.\\

If $\beta > \gamma$,  then $\lim\limits_{t \rightarrow \infty} \displaystyle \frac{m_{2}^{I}(t,x,x)}{(m_{1}^{I}(t,x))^{2}} \rightarrow E_{1}$ where $E_{1} = \rho_{0} - \Big(\displaystyle \frac{\beta + \gamma + 2\kappa}{-\beta + \gamma}\Big)$.\\

If $\beta=\gamma$, then $\lim\limits_{t \rightarrow \infty} \displaystyle \frac{m_{2}^{I}(t,x,x)}{(m_{1}^{I}(t,x))^{2}} \rightarrow \infty$.\\

\textbf{When $x \neq y$ in Homogeneous Space:}\\
Substituting the equations from Theorems \ref{thm:first moment homogeneous space} and Theorem \ref{thm: 2nd moment} into $\displaystyle \frac{m_{2}^{I}(t,x,y)}{m_{1}^{I}(t,x)m_{1}^{I}(t,y)}$ and simplifying the equation we get
\begin{center}
$\displaystyle \frac{m_{2}^{I}(t,x,y)}{m_{1}^{I}(t,x)m_{1}^{I}(t,y)} = \rho_{0} + \Big(\displaystyle \frac{2\kappa a(v)}{\beta - \gamma}\Big) \Big[ e^{\displaystyle -(\beta-\gamma)t} - 1 \Big]$
\end{center}
where $v=\|x-y\|$.
Denote $ E_{2} = \rho_{0} - \Big(\displaystyle \frac{ 2\kappa a(v)}{\beta - \gamma}\Big)$ and the intermittency analysis following a similar process as the case when $x=y$. The results of the analysis are summarized in Table \ref{tab:intermittency_analysis_merged}.

In the inhomogeneous space, we note that the transition probability $p(t, x, y)$ is given by
$$
p(t, x, y)=\frac{1}{(2 \pi)^d} \int_{T^d} e^{i k(x-y)} e^{\kappa \hat{\mathcal{L}}(k) t} d k=\frac{1}{(2 \pi)^d} \int_{T^d} \cos (k(x-y)) \hat{p}(t, 0, k) d k
$$

Since $\cos (k(x-y)) \leq 1$ for all $k$, it follows that
$$
p(t, x, y) \leq \frac{1}{(2 \pi)^d} \int_{T^d} \hat{p}(t, 0, k) d k=p(t, 0,0)
$$

Thus, we have $p(t, x, y) \leq p(t, 0,0)$.

In the large deviation analysis by \citep{19}, limit theorems on the asymptotic behavior of transition probabilities were established. Specifically, as $t \rightarrow \infty$, they proved that
$$
p(t, 0,0)=\frac{E_3}{t^{d / 2}}+o\left(t^{-d / 2}\right) \rightarrow 0
$$
where $E_3 \in \mathbb{R}$ is a constant. This result indicates that $p(t, 0,0)$ decays to zero as time progresses, which has significant implications for the analysis of intermittency in our model.

\textbf{When $x=y$ in Inhomogeneous Space:}\\
Recall $m_{1}^{I}(t,x)$ from Theorem \ref{thm: first moment inhomogeneous space} and $m_{2}^{I}(t,x,x)$ from Theorem \ref{thm: 2nd moment} into and substituting into the equation $\displaystyle \frac{m_{2}^{I}(t,x,x)}{m_{1}^{I}(t,x)^2}$ we get
\begin{center}
\begin{align*}
\displaystyle \frac{m_{2}^{I}(t,x,x)}{(m_{1}^{I}(t,x))^{2}} &=
\frac{\rho_{0}}{p(t,0,x)} + \Big(\frac{\kappa (e^{ 3(\beta-\gamma)t}-1)}{3(\beta-\gamma)(p(t,0,x))^{2}}\Big) + \Big( \frac{\beta+\gamma+2\kappa}{3(\beta-\gamma)p(t,0,x)}\Big)\Big[e^{ 3(\beta-\gamma)t} -1 \Big]\\
& \ge
\frac{\rho_{0}}{p(t,0,0)} + \Big(\frac{\kappa(e^{ 3(\beta-\gamma)t}-1)}{3(\beta-\gamma)(p(t,0,0))^{2}}\Big)  + \Big( \frac{\beta+\gamma+2\kappa}{3(\beta-\gamma)p(t,0,0)}\Big) \cdot
\Big[e^{ 3(\beta-\gamma)t} -1 \Big].
\end{align*}
\end{center}

\textbf{When $x \neq y$ in Inhomogeneous Space:}\\
Using the solutions $m_{1}^{I}(t,x) = e^{\displaystyle (\beta-\gamma)t} p(t,0,x)$\\ and $m_{1}^{I}(t,y) = e^{\displaystyle (\beta-\gamma)t} p(t,0,y)$ from Theorem \ref{thm: first moment inhomogeneous space} and $m_{2}^{I}(t,x,y)$ from Theorem \ref{thm: 2nd moment}, we have that
\begin{align*}
\displaystyle \frac{m_{2}^{I}(t,x,y)}{m_{1}^{I}(t,x)m_{1}^{I}(t,y)}& = \displaystyle \frac{\rho_{0}}{p(t,0,y)} - \displaystyle \Big(\frac{2\kappa a(v)}{3(\beta-\gamma)p(t,0,y)}\Big) \Big[e^{\displaystyle 3(\beta-\gamma)t}-1 \Big] \\
&\geq \displaystyle \frac{\rho_{0}}{p(t,0,0)} - \displaystyle \Big(\frac{2\kappa a(v)}{3(\beta-\gamma)p(t,0,0)}\Big) \Big[e^{\displaystyle 3(\beta-\gamma)t}-1 \Big].
\end{align*}

As $t \rightarrow \infty$, the term $p(t, 0,0) \rightarrow 0$, causing denominators involving $p(t, 0,0)$ to approach zero and the expression $\frac{m_2^I(t, x, x)}{\left(m_1^I(t, x)\right)^2}$ to tend toward infinity. This unbounded growth indicates that the variance of the infected population at location $x$ becomes significantly larger than the square of its mean, signifying the presence of intermittency—substantial fluctuations relative to the mean leading to infection clusters at specific locations.

Thus, the asymptotic behavior of $p(t, 0,0)$ critically influences the higher moments of the infected population, underscoring the pivotal role of mobility in the spread and clustering of infections in inhomogeneous spaces. While $p(t, 0,0)$ diminishes over time, reflecting dispersal, the increasing ratio highlights the tendency of infections to concentrate despite this dispersion, offering valuable insights into the mechanisms driving epidemic clustering in spatially variable environments.

\begin{table}[ht]
\centering
\begin{tabular}{ | c | c | c | c | c | }
\hline
\textbf{Space Type} & \textbf{Condition} & $\displaystyle \lim_{t \rightarrow \infty} \frac{m_{2}^{I}(t,x,x)}{(m_{1}^{I}(t,x))^{2}}$ & $\displaystyle \lim_{t \rightarrow \infty} \frac{m_{2}^{I}(t,x,y)}{m_{1}^{I}(t,x)m_{1}^{I}(t,y)}$ & \textbf{Result} \\
\hline
\multirow{3}{*}{Inhomogeneous} & $\beta < \gamma$ &  $\rightarrow \infty$ & $\rightarrow \infty$ & Intermittency \\ 
& $\beta = \gamma$ & $\rightarrow \infty$ & $\rightarrow \infty$ & Intermittency\\
& $\beta >  \gamma$ & $\rightarrow \infty$ & $\rightarrow \infty$ & Intermittency \\
\hline
\multirow{2}{*}{Homogeneous} & $\beta \leq \gamma$ &  $\rightarrow \infty$ & $\rightarrow \infty$ & Intermittency \\ 
& $\beta >  \gamma$ & $\rightarrow E_{1} < \infty$ & $\rightarrow E_{2} < \infty$ & No Intermittency \\
\hline
\end{tabular}
\caption{Intermittency Analysis in Homogeneous and Inhomogeneous Spaces.}
\label{tab:intermittency_analysis_merged}
\end{table}

In homogeneous space, the dynamics of infection spread depend significantly on the relationship between the infection rate ($\beta$) and recovery rate ($\gamma$). When $\beta>\gamma$, the infection spreads uniformly across the population, without forming clusters, as the infection dominates. In contrast, when $\beta\leq \gamma$, intermittency occurs due to random variations in individual movements. Over time, this leads to the formation of clusters where infections persist longer in certain regions due to random factors such as travel patterns or local interactions. This reflects early epidemic behavior where localized outbreaks emerge.

In inhomogeneous space, spatial variability plays a crucial role in shaping infection dynamics. Geographic and demographic differences cause infections to spread more rapidly in some regions while remaining contained in others, even when $\gamma > \beta$. Clustering or intermittency can still occur regardless of the relationship between $\beta$ and $\gamma$, reflecting real-world scenarios. The interplay of infection rate, recovery rate, and mobility patterns contributes to the formation of infection clusters. In real-world scenarios, urban areas with high population density and frequent movement may experience localized outbreaks, while rural areas with limited movement may see isolated, persistent infections. This pattern was evident in past pandemics, where certain regions were heavily impacted, while others were shielded by local policies, infrastructure, or natural barriers.

The analysis of intermittency in both homogeneous and inhomogeneous spaces offers valuable insights into how spatial movements influence the spread of infectious diseases. This has practical implications for public health interventions, where understanding how clusters form and persist can help in targeting high-risk areas with more focused measures, such as localized lockdowns, vaccination drives, or resource allocation.

\section{Discussion}

In this paper, we introduced a novel SIR model incorporating mobility, allowing particles to move spatially within the susceptible, infected, and recovered groups along with the usual inter-compartmental transitions. By analyzing the first moments, we found that under certain conditions, the infected population reaches a steady state rather than vanishing, which contrasts with the classical SIR model. This steady state arises due to continuous movement between locations, ensuring the infected population does not necessarily diminish when $\beta = \gamma$. In our mobility-enhanced model, both local and long-distance movements governed by the mobility kernel can lead to different outcomes, which are determined by the newly defined basic reproduction number $R_0^m$. In homogeneous space, $R_0^m$ is reduced to the classical reproduction number $R_0=\displaystyle\frac{\beta}{\gamma}$. However, in inhomogeneous space, we encounter a more interesting case, where $R_0^m=\displaystyle\frac{\kappa \hat{a}(k)+\beta}{\gamma}=R_0+\frac{\kappa \hat{a}(k)}{\gamma}$, with $\hat{a}(k)$ being the Fourier transform of the mobility kernel $a(z)$. In both homogeneous and inhomogeneous spaces, if $R_0^m=1$, the infected population will reach a steady state in the long term. If $R_0^m<1$, the infected population will eventually vanish. When $R_0^m>1$: (1) In homogeneous space, the infected population will grow exponentially. (2) In inhomogeneous space, an outbreak will occur and spread across locations if the mobility measure $|\kappa \hat{a}|>0$. However, if the mobility measure is zero ($|\kappa \hat{a}|=0$, i.e., $R_0^m=R_0$), indicating no mobility, the outbreak will be confined to the initial location, and other locations will remain unaffected. This analysis highlights the critical role of mobility in the spread of infections in inhomogeneous spaces, which are more akin to our everyday environment. The basic reproduction number $R_0^m$ determines the potential for an outbreak, but the mobility effect strength $|\kappa \hat{a}|$ influences whether the infection spreads to other locations or remains isolated. This framework allows for the possibility of isolating outbreaks or facilitating their spread across regions, which is a feature not accounted for in the classical SIR model.

Furthermore, the intermittency phenomenon, analyzed through the second moment, demonstrates that infection clusters form under specific conditions. We observed that intermittency occurs when $\beta < \gamma$ and $\beta = \gamma$ in homogeneous space, and in all cases  in inhomogeneous space. This provides a mathematical explanation for understanding the clustering behavior seen during pandemics, such as the higher concentration of infections in large cities like Los Angeles and New York City during COVID-19. 

The mobility-based general nonlocal operator in our model opens up numerous possibilities for future research. First, the current assumption that healthy and infected individuals have the same mobility does not always reflect reality, as infected individuals often experience impaired mobility. Future models could account for differing mobility rates, increasing the accuracy of real-world epidemic simulations. Another extension involves generalizing the kernel $a(z)$ to study various movement types, such as asymmetric and various types of distributions, to represent both local and long-distance movements more realistically. and estimate $a(z)$ by the real data.

Additionally, expanding this model from discrete spatial settings like $\mathbb{Z}^d$ to continuous spaces $\mathbb{R}^d$ would enable more complex modeling of population spread. In $\mathbb{R}^d$, spatial movement could be modeled through diffusion processes represented by partial differential equations (PDEs), such as reaction-diffusion systems. These systems would more accurately capture large-scale epidemics by considering both local and nonlocal movements and addressing challenges like population density variations and geographic barriers.

\backmatter

\section*{Declarations}
\begin{itemize}
\item Funding

Dan Han was supposed by University of Louisville EVPRI grant ``Spatial Population Dynamics with Disease" and AMS Research Communities ``Survival Dynamics for Contact Process with Quarantine".

\item Data availability 

Data sharing not applicable to this article as no datasets were generated or analysed during
the current study.

\item Author Contribution

Ciana Applegate: Writing - review \& editing, initial draft, formal analysis.\\
Dan Han: Writing - review \& editing, original draft, formal analysis, supervision, conceptualization.\\
Jiaxu Li: Writing - review \& editing.
\end{itemize}

\section{Compliance with Ethical Standards}
\begin{itemize}
    \item Disclosure of potential conflicts of interest.
    
    The authors declare that they have no conflict of interest or competing interests relevant to this research.
\item Research involving Human Participants and/or Animals

This article does not contain any studies or data involving human participants and animals performed by any of the authors.

\end{itemize}





%%common bib file

%\printbibliography

\bibliography{mybib}

\end{document}